\documentclass[a4paper,10pt]{amsart}
\usepackage[top=0.8in, bottom=0.8in, left=1.25in, right=1.25in]{geometry}
\usepackage{fancybox}
\usepackage{amsmath}
\usepackage{amssymb}
\usepackage{amsthm}
\usepackage{mathrsfs}
\usepackage{epstopdf}
\usepackage{epstopdf}
\usepackage{array} 
\usepackage{amscd}
\usepackage{color}
\usepackage{hyperref}

\usepackage[all]{xypic}

\usepackage{amsfonts}
\usepackage{amssymb}
\usepackage{mathrsfs}

\DeclareFontFamily{U}{rsfs}{%
\skewchar\font127}
\DeclareFontShape{U}{rsfs}{m}{n}{%
<-6>rsfs5<6-8.5>rsfs7<8.5->rsfs10}{}
\DeclareSymbolFont{rsfs}{U}{rsfs}{m}{n}
\DeclareSymbolFontAlphabet
{\mathrsfs}{rsfs}
\DeclareRobustCommand*\rsfs{%
\@fontswitch\relax\mathrsfs}

\theoremstyle{plain}
\newtheorem{thm}{Theorem}[section]
\newtheorem{prop}[thm]{Proposition}
\newtheorem{lem}[thm]{Lemma}

\newtheorem{prop-defi}[thm]{Proposition-Definition}
\newtheorem{thm-defi}[thm]{Theorem-Definition}
\newtheorem{lem-defi}[thm]{Lemma-Definition}

\newtheorem{conj}[thm]{Conjecture}
\newtheorem{exam}[thm]{Example}

\newdimen\argwidth
\def\db[#1\db]{
 \setbox0=\hbox{$#1$}\argwidth=\wd0
 \setbox0=\hbox{$\left[\box0\right]$}
  \advance\argwidth by -\wd0
 \left[\kern.3\argwidth\box0 \kern.3\argwidth\right]}

\newcommand{\cC}{\mathcal{C}}

\newcommand{\eE}{\mathcal{E}}
\newcommand{\fF}{\mathcal{F}}

\newcommand{\hH}{\mathcal{H}}
\newcommand{\iI}{\mathcal{I}}

\newcommand{\lL}{\mathcal{L}}

\newcommand{\oO}{\mathcal{O}}

\newcommand{\dR}{\mathbf{R}}
\newcommand{\dL}{\mathbf{L}}

\newcommand{\Pic}{\mathop{\rm Pic}\nolimits}

\newcommand{\id}{\textrm{id}}

\newcommand{\ch}{\mathop{\rm ch}\nolimits}

\newcommand{\Ext}{\mathop{\rm Ext}\nolimits}

\newcommand{\Coh}{\mathop{\rm Coh}\nolimits}

\newcommand{\cneq}{\mathrel{\raise.095ex\hbox{:}\mkern-4.2mu=}}
\newcommand{\eqcn}{\mathrel{=\mkern-4.5mu\raise.095ex\hbox{:}}}

\newcommand{\DT}{\mathop{\rm DT}\nolimits}

\newcommand{\RHom}{\mathop{\dR\mathrm{Hom}}\nolimits}

\makeatletter
 
  \@addtoreset{equation}{section}
\makeatother

\title[{GV type invariants on CY 4-folds via 
descendent insertions}]
{Gopakumar-Vafa type invariants on Calabi-Yau 4-folds \\
via descendent insertions}
\date{}
\author{Yalong Cao}
\address{Kavli Institute for the Physics and Mathematics of the Universe (WPI),The University of Tokyo Institutes for Advanced Study, The University of Tokyo, Kashiwa, Chiba 277-8583, Japan}
\email{yalong.cao@ipmu.jp}
\author{Yukinobu Toda}
\address{Kavli Institute for the Physics and Mathematics of the Universe (WPI),The University of Tokyo Institutes for Advanced Study, The University of Tokyo, Kashiwa, Chiba 277-8583, Japan}
\email{yukinobu.toda@ipmu.jp}

\begin{document}
\maketitle
\begin{abstract}
The Gopakumar-Vafa type invariants on Calabi-Yau 4-folds
(which are non-trivial only for genus zero and one)
are defined by Klemm-Pandharipande from 
Gromov-Witten theory,
and their integrality is conjectured. 
In a previous work of Cao-Maulik-Toda, $\DT_4$ invariants with primary insertions on moduli spaces of 
one dimensional stable sheaves are used to give a sheaf theoretical interpretation of the genus zero GV type invariants.  
In this paper, we propose a sheaf theoretical interpretation of the genus one 
GV type invariants using descendent insertions on the above moduli spaces. 
The conjectural formula in particular implies nontrivial constraints on genus zero GV type (equivalently GW) invariants of CY 4-folds 
which can be proved by the WDVV equation. 
\end{abstract}



\section{Introduction}
\subsection{Background on GV invariants of CY 3-folds}\label{sect on GV on 3-folds}
Let $Y$ be a smooth projective Calabi-Yau 3-fold over $\mathbb{C}$.
For each $g \in \mathbb{Z}_{\geqslant 0}$ and $\beta \in H_2(Y, \mathbb{Z})$, 
the Gromov-Witten invariant 
\begin{align*}
\mathrm{GW}_{g, \beta} =\int_{[\overline{M}_g(Y, \beta)]^{\rm{vir}}}
1 \in \mathbb{Q}
\end{align*}
 enumerates 
stable maps $f \colon C \to Y$ with 
$f_{\ast}[C]=\beta$ and $g(C)=g$. 
Because of multiple cover phenomena,
the invariants $\mathrm{GW}_{g, \beta}$ are not necessarily integers, 
so may not be regarded as ideal curve counting. 
However if we write the generating series of GW invariants as 
\begin{align}\label{intro:GV3}
\sum_{\beta>0, g \geqslant 0}
\mathrm{GW}_{g, \beta}\lambda^{2g-2}t^{\beta}=
\sum_{\beta>0, g\geqslant 0, k\geqslant 1}\frac{n_{g, \beta}}{k}
\left( 2\sin\left( \frac{k\lambda}{2} \right) \right)^{2g-2} t^{k\beta},
\end{align}
for some $n_{g, \beta} \in \mathbb{Q}$
(called \textit{Gopakumar-Vafa invariants}),  
then $n_{g, \beta}$ are expected to be 
integers. 
The above integrality conjecture for $n_{g, \beta}$ 
is proved by Ionel-Parkar~\cite{IP} via symplectic geometry. 

Furthermore there should be a sheaf theoretical interpretation 
of the integrality of $n_{g,\beta}$. 
From the perspective 
of Gopakumar-Vafa~\cite{GV} on type IIA and M-theory duality, 
the invariants $n_{g, \beta}$ are 
expected to be related to characters of  
$\mathfrak{sl}_2 \times \mathfrak{sl}_2$-actions on 
some cohomology theories of 
moduli spaces $M_1(Y, \beta)$
of one dimensional stable sheaves $F$ on 
$Y$ with $[F]=\beta$ and $\chi(F)=1$. 
For genus zero, 
Katz~\cite{Katz} 
conjectured that the invariant $n_{0, \beta}$ 
equals to 
the degree of the virtual fundamental class of 
$M_1(Y, \beta)$, which is 
a special case 
of Donaldson-Thomas invariants. 
For higher genus, 
Maulik-Toda~\cite{MT} conjectured the following identity
(which itself modifies earlier approaches by Hosono-Saito-Takahashi~\cite{HST} and 
Kiem-Li~\cite{KL}):
\begin{align}\label{intro:MT}
\sum_{i \in \mathbb{Z}} \chi(^{p} \hH^i (\dR \pi_{\beta\ast}\phi_M))y^i
=\sum_{g\geqslant 0} n_{g, \beta}(y^{\frac{1}{2}}+y^{-\frac{1}{2}})^{2g}. 
\end{align}
Here $\phi_{M}$ is a perverse sheaf on $M_{1}(Y, \beta)$
locally isomorphic to vanishing cycle sheaves of local Chern-Simons 
functions, 
$\pi_{\beta} \colon M_1(Y, \beta) \to \mathrm{Chow}_{\beta}(Y)$ is 
the Hilbert-Chow map, and $^{p}\hH^i(-)$ is the $i$-th perverse cohomology.
See~\cite[Section~7.1]{MT} for
 the relationship of the formula (\ref{intro:GV3})
with $\mathfrak{sl}_2 \times \mathfrak{sl}_2$-action. 
The conjectural equality~\eqref{intro:MT} is widely open (see ~\cite[Section~9]{MT}).

\subsection{GV type invariants of CY 4-folds}
In what follows, we discuss similar stories 
for Calabi-Yau 4-folds. 
Let $X$ be a smooth projective Calabi-Yau 4-fold\footnote{In this paper, a Calabi-Yau 4-fold is a complex smooth projective 4-fold $X$ satisfying $K_X \cong \oO_X$.}. 
As an analogy of 
GV invariants on Calabi-Yau 3-folds defined by (\ref{intro:GV3}), 
Klemm-Pandharipande \cite{KP} defined \textit{Gopakumar-Vafa type invariants}
on $X$:
\begin{align}\label{intro GV type invs}n_{0,\beta}(\gamma)\in \mathbb{Q}, \quad n_{1,\beta}\in \mathbb{Q}, \end{align}
from corresponding primary GW invariants (see Section \ref{section on GV/GW} for details). 
Here an insertion $\gamma\in H^4(X,\mathbb{Z})$ is needed in the definition of genus zero invariants, and  
the $g\geqslant 2$ invariants are defined to be zero
because the virtual dimensions of corresponding GW moduli
spaces are negative.
The GV type invariants (\ref{intro GV type invs})
are conjectured to be integers and verified in many examples from computations of GW invariants \cite{KP}.
The genus zero integrality conjecture has been proved by Ionel-Parker \cite[Theorem 9.2]{IP} using symplectic geometry.

Motivated by the 
conjectural identity (\ref{intro:MT})
for CY 3-folds, it is desirable to
give a sheaf theoretical interpretation of GV type 
invariants (\ref{intro GV type invs}) for CY 4-folds 
in terms of moduli spaces $M_1(X, \beta)$
of one dimensional stable sheaves $F$ on $X$
with $[F]=\beta$ and $\chi(F)=1$. 

For $\gamma \in H^4(X, \mathbb{Z})$, let 
$\tau_0(\gamma)$ be the primary insertion:
\begin{align*}  
\tau_0(\gamma):=(\pi_{M})_{\ast}(\pi_X^{\ast}\alpha \cup\ch_{3}(\mathbb{F})) \in H^2(M_1(X, \beta), \mathbb{Z}).
\end{align*}
Here
$\mathbb{F}$ is a universal 
sheaf, and $\pi_X$, $\pi_M$ are the projections from 
$X \times M_1(X, \beta)$ onto corresponding factors. 
In~\cite{CMT1},  
Maulik and the authors  
proposed 
the following conjecture,
which gives a sheaf theoretical interpretation 
of genus zero GV type invariants on CY 4-folds
and could be regarded as an analogue of Katz conjecture~\cite{Katz}.
 
\begin{conj}\label{intro g=0 one dim sheaf conj}\emph{(\cite[Conjecture 0.2]{CMT1})}
For $\gamma \in H^4(X, \mathbb{Z})$, let 
$\langle\tau_0(\gamma) \rangle_{\beta}$
be the $\DT_4$ invariant with primary insertions on 
$M_1(X, \beta)$, i.e.
\begin{align*}
\langle \tau_0(\gamma)\rangle_{\beta} := 
\int_{[M_1(X, \beta)]^{\rm{vir}}}\tau_0(\gamma). 
\end{align*}
Then for a certain choice of an orientation on $M_1(X, \beta)$, we have 
\begin{align}\label{conj:g=0:GV} \langle\tau_0(\gamma) \rangle_{\beta}=n_{0,\beta}(\gamma). \end{align}
Here $n_{0,\beta}(\gamma)$ is the g=0 Gopakumar-Vafa type invariant \eqref{intro GV type invs}.
\end{conj}

Here $[M_1(X, \beta)]^{\rm{vir}}\in H_2(M_1(X, \beta),\mathbb{Z})$
is the $\DT_4$-virtual class of $M_1(X, \beta)$, determined by a choice of an orientation \cite{CGJ}. 
$\DT_4$-virtual classes on moduli spaces of stable sheaves 
on CY 4-folds are constructed in general by 
Borisov-Joyce~\cite{BJ} via derived
algebraic/$C^{\infty}$-geometry, and in some special cases by Cao-Leung~\cite{CL1}
via gauge theory and classical algebraic geometry (see Section \ref{subsec:DT4} for a brief review). 
We refer to~\cite{CGJ, CK1, CK2, CKM, CKM2, CL2, CMT2, CT19, CT20a, CT20b}
for some recent developments of such theory. 
As we discuss below, the main purpose of this paper is 
to use $\DT_4$-theory together with 
descendent insertions to give a conjectural sheaf theoretical
interpretation of genus one GV type invariants. 


\subsection{Main conjecture}
As we explained in Section~\ref{sect on GV on 3-folds}, 
in order to give a sheaf theoretical interpretation of 
higher genus GV invariants on CY 3-folds, 
the theory of perverse 
sheaves of vanishing cycles plays a key role.
However, such a theory is not feasible on CY 4-folds as moduli spaces of sheaves are no longer locally written as critical locus of functions on Zariski tangent spaces. 

Our proposal is to use descendent $\DT_4$ invariants to capture 
 $g=1$ GV type invariants on CY 4-folds.
For $\alpha\in H^2(X,\mathbb{Z})$, the descendent insertion is 
defined by 
\begin{align*}  
\tau_1(\alpha):=(\pi_{M})_{\ast}(\pi_X^{\ast}\alpha \cup\ch_{4}(\mathbb{F}) ) \in H^2(M_1(X, \beta), \mathbb{Q}).
\end{align*}
A subtlety here is that a universal sheaf 
$\mathbb{F}$ is not unique, and 
$\ch_4(\mathbb{F})$ may depend on its choice. 
We choose $\mathbb{F}$ to be the normalized one \eqref{norm univ sheaf}
in the sense that the determinant of 
$\dR \pi_{M\ast}\mathbb{F}$ is trivial. 
The descendent $\DT_4$ invariants
are defined by 
\begin{align*} 
\langle\tau_1(\alpha) \rangle_{\beta}:=\int_{[M_{1}(X,\beta)]^{\rm{vir}}}\tau_1(\alpha). \end{align*}
The following is our main conjecture of this paper: 
\begin{conj}\emph{(Conjecture \ref{main conj})} \label{intro main conj}
For a certain choice of an orientation on $M_1(X, \beta)$, 
there is an equality of functions  
of $\alpha \in H^2(X)$
\begin{align}\label{conj:GV:g=1}
\langle\tau_1(\alpha) \rangle_{\beta}=\frac{n_{0,\beta}(\alpha^2)}{2\,(\alpha\cdot\beta)}
-\sum_{\beta_1+\beta_2=\beta}\frac{(\alpha\cdot\beta_1)(\alpha\cdot\beta_2)}{4\,(\alpha\cdot\beta)}m_{\beta_1,\beta_2}
-\sum_{k\geqslant1,\,k|\beta}\frac{(\alpha\cdot\beta)}{k}\,n_{1,\beta/k}.
 \end{align}
Here $n_{0,\beta}(-)$ and $n_{1,\beta}$ are genus 0 and 1 GV type invariants \eqref{intro GV type invs} and $m_{\beta_1,\beta_2}$ are meeting invariants which can be inductively determined by genus 0 GV type invariants \emph{(see \eqref{meeting invs})}.
\end{conj}
Note that 
by the induction on the divisibility of $\beta$
together with the conjectural 
identity (\ref{conj:g=0:GV}), the
formula (\ref{conj:GV:g=1})
in principle 
determines $n_{1, \beta}$ 
in terms of 
primary $\DT_4$-invariants
$\langle \tau_0(\gamma) \rangle_{\beta}$
and descendent $\DT_4$-invariants 
$\langle \tau_1(\alpha)\rangle_{\beta}$. 
Therefore it gives a sheaf theoretical interpretation of 
$n_{1, \beta}$ from $M_1(X, \beta)$, in a different flavor
from (\ref{intro:MT}). 

Also note that the LHS of the formula (\ref{conj:GV:g=1}) is linear on 
$\alpha$, but the RHS is a priori rational function on $\alpha$. 
In order to make sense of the formula (\ref{conj:GV:g=1}), 
we need to show the absence of pole
in the RHS, which requires the following constraint 
\begin{align}\label{intro constraint}
n_{0,\beta}(\alpha^2)=\frac{1}{2}\sum_{\begin{subarray}{c}\beta_1+\beta_2=\beta  \\ \beta_1, \beta_2>0 \end{subarray}}(\alpha\cdot\beta_1)(\alpha\cdot\beta_2)\, m_{\beta_1,\beta_2}, 
\end{align}
for any $\alpha \in H^2(X)$ with $\alpha \cdot \beta=0$. This equality can be equivalently written 
using GW invariants (Proposition \ref{prop:GW}) and gives nontrivial constraints on GW invariants of Calabi-Yau 4-folds whose Picard numbers are bigger than one (Example \ref{exam on constrain}).
We prove that the above equality can be derived from the WDVV equation \cite{W,DVV}. 
\begin{thm}\emph{(Theorem \ref{constraint holds})}\label{intro constraint holds}
The identity \eqref{intro constraint} holds for any Calabi-Yau 4-fold.
In particular, the formula in Conjecture \ref{intro main conj} makes 
sense as an identity of linear functions on $\alpha\in H^2(X)$.
\end{thm}
Our Conjecture \ref{intro main conj} 
is written down using a heuristic argument in an ideal $\mathrm{CY}_4$ geometry discussed in Section \ref{heuristic argue}.
At the moment, we are not able to fix the signs in front of genus one GV type invariants solely from the heuristic argument.
The minus sign is in fact experimentally obtained by explicit computations of several examples in Section \ref{sect on examples}.

\subsection{Verifications of Conjecture~\ref{intro main conj} in examples}
We verify our main conjecture 
in several examples.  
The first example is an elliptic fibered CY 4-fold
given by a Weierstrass model. 
\begin{thm}\emph{(Theorem \ref{thm:elliptic})}
Let $\pi \colon  X \to \mathbb{P}^3$ be an elliptic fibered Calabi-Yau 4-fold \eqref{elliptic CY4}
given by a Weierstrass model,
 and $f=\pi^{-1}(p)$ be a generic fiber.
Then Conjecture \ref{intro main conj} holds for $\beta=r[f]$ with $r\geqslant1$.
\end{thm}
The proof of this case relies on a careful description of the normalized universal sheaf by Fourier-Mukai transforms 
(see Lemma \ref{lem:spherical}$-$\ref{lem:Ktheory}).

The next example is the product of a CY 3-fold and an elliptic curve. 
\begin{prop}\emph{(Proposition \ref{prop on product})}
Let $X=Y\times E$ be the product of a smooth projective Calabi-Yau 3-fold $Y$ with an elliptic curve $E$. Then 
Conjecture \ref{intro main conj} holds in the following cases
\begin{itemize}
\item $\beta=r[E]$ with $r\geqslant 1$, 
\item any $\beta\in H_2(Y,\mathbb{Z})\subset H_2(X,\mathbb{Z})$.
\end{itemize}
\end{prop}

The final examples are 
non-compact ones for small degree curve classes.
\begin{thm}\emph{(Proposition \ref{prop on local P2}, \ref{prop on local P1 product}, \ref{local Fano 3})}\label{intro:locurve}
Conjecture \ref{intro main conj} holds in the following cases: 
\begin{itemize}
\item $X=\mathrm{Tot}_{\mathbb{P}^2}(\oO(-1)\oplus \oO(-2))$ and $\beta=d\,[\mathbb{P}^1]\in H_2(X,\mathbb{Z})$ with $d\leqslant 3$.
\item $X=\mathrm{Tot}_{\mathbb{P}^1\times \mathbb{P}^1}(\oO(-1,-1)\oplus \oO(-1,-1))$ and $\beta=(d_1,d_2)\in H_2(X,\mathbb{Z})$
with $d_1,d_2\leqslant 2$.
\item $X=K_{\mathbb{P}^3}$ and $\beta=d\,[\mathbb{P}^1]\in H_2(X,\mathbb{Z})$ with $d\leqslant 3$.
\end{itemize}
\end{thm}
A common feature of the examples in Theorem~\ref{intro:locurve}
is that the CY 4-fold $X$ can be written as the total 
space of canonical bundle $K_Y$ of a (possibly non-compact) Fano 3-fold $Y$, and there is an isomorphism
\begin{align*}
i_{\ast} \colon M_1(Y, \beta) \stackrel{\cong}{\to} M_1(X,\beta),
\end{align*}
where $i \colon Y \hookrightarrow X$ is the zero section. 
Remarkably in this case, the $\DT_4$ virtual class has a preferred choice of orientations:
\begin{align}\label{prefer choice}[M_1(X,\beta)]^{\mathrm{vir}}=(-1)^{c_1(Y)\cdot\beta-1}\cdot [M_1(Y,\beta)]^{\mathrm{vir}},  \end{align}
where the sign in the RHS is conjectured to give the correct choice of orientations which makes Conjecture \ref{intro g=0 one dim sheaf conj} true (see \cite[Conjecture 0.2]{CaoFano}). We verify Conjecture \ref{intro main conj} in the above examples using this particular choice of orientations. 

Finally we remark that the proof of the $d=3$ case for $X=K_{\mathbb{P}^3}$ requires an identification of the moduli space $M_1(X,3)$ with
the moduli space $P_1(X,3)$ of PT stable pairs (see Proposition \ref{local Fano 3}).   
By using the 4-fold stable pair vertex \cite{CK2, CKM}, one can explicitly compute the invariant and obtain a matching. 
We thank Sergej Monavari for doing a computation for us using his Maple program. 
Other cases can be computed through explicit descriptions of moduli spaces and virtual classes, and are reduced 
to calculations using classical intersection theory.

\subsection{Acknowledgement}
We are grateful to Martijn Kool and Sergej Monavari for helpful discussions and 
warmly thank Sergej Monavari for his help in doing a computation using his program.
This work is partially supported by the World Premier International Research Center Initiative (WPI), MEXT, Japan.
Y. C. is partially supported by JSPS KAKENHI Grant Number JP19K23397 and Newton International Fellowships Alumni 2019.
Y. T. is partially supported by Grant-in Aid for Scientific Research grant (No. 26287002) from MEXT, Japan.

\section{Definitions and conjectures}

\subsection{GV type invariants of CY 4-folds}\label{section on GV/GW}
Let $X$ be a Calabi-Yau 4-fold.
The genus 0 Gromov-Witten invariants on $X$ are defined using
insertions: for $\gamma \in H^{4}(X, \mathbb{Z})$,
one defines
\begin{equation}
\mathrm{GW}_{0, \beta}(\gamma)
=\int_{[\overline{M}_{0, 1}(X, \beta)]^{\rm{vir}}}\mathrm{ev}^{\ast}(\gamma)\in\mathbb{Q},
\nonumber \end{equation}
where $\mathrm{ev} \colon \overline{M}_{0, 1}(X, \beta)\to X$
is the evaluation map.

The \textit{genus 0 Gopakumar-Vafa type invariants} 
\begin{align}\label{g=0 GV}
n_{0, \beta}(\gamma) \in \mathbb{Q}
\end{align}
are defined by Klemm-Pandharipande \cite{KP} from the identity
\begin{align*}
\sum_{\beta>0}\mathrm{GW}_{0, \beta}(\gamma)q^{\beta}=
\sum_{\beta>0}n_{0, \beta}(\gamma) \sum_{d=1}^{\infty}
d^{-2}q^{d\beta}.
\end{align*}
For the genus 1 case, virtual dimensions of moduli spaces of stable maps are zero, so Gromov-Witten invariants
\begin{align*}
\mathrm{GW}_{1, \beta}=\int_{[\overline{M}_{1, 0}(X, \beta)]^{\rm{vir}}}
1 \in \mathbb{Q}
\end{align*}
can be defined
without insertions.
The \textit{genus 1 Gopakumar-Vafa type invariants}
\begin{align}\label{g=1 GV}
n_{1, \beta} \in \mathbb{Q}
\end{align}
 are defined in~\cite{KP} by the identity
\begin{align*}
\sum_{\beta>0}
\mathrm{GW}_{1, \beta}q^{\beta}=
&\sum_{\beta>0} n_{1, \beta} \sum_{d=1}^{\infty}
\frac{\sigma(d)}{d}q^{d\beta}
+\frac{1}{24}\sum_{\beta>0} n_{0, \beta}(c_2(X))\log(1-q^{\beta}) \\
&-\frac{1}{24}\sum_{\beta_1, \beta_2}m_{\beta_1, \beta_2}
\log(1-q^{\beta_1+\beta_2}),
\end{align*}
where $\sigma(d)=\sum_{i|d}i$ and $m_{\beta_1, \beta_2}\in\mathbb{Z}$ are called meeting invariants defined as follows.

Take a basis $S_1, \cdots, S_k$ of the free part of $H^4(X, \mathbb{Z})$ and let
\begin{align}\label{kunneth decom}
\sum_{i, j} g^{ij} [S_i \otimes S_j] \in H^8(X \times X, \mathbb{Z})
\end{align}
be the K\"unneth
decomposition of the $(4,4)$ component of the diagonal class (mod torsion).
For $\beta_1, \beta_2 \in H_2(X, \mathbb{Z})$,
the \emph{meeting invariant}\,\footnote{They are integers because of Ionel-Parker's proof of genus zero integrality \cite[Theorem 9.2]{IP}.}
\begin{align}\label{meeting invs}
m_{\beta_1, \beta_2} \in \mathbb{Z}
\end{align}
is the virtual number of rational curves
of class $\beta_1$ meeting rational curves of
class $\beta_2$. They are uniquely determined by
the following rules: \\

(i) The meeting invariants are symmetric,
$m_{\beta_1, \beta_2}=m_{\beta_2, \beta_1}$. \\

(ii) If either $\deg(\beta_1)\leqslant 0$ or $\deg(\beta_2)\leqslant 0$,
we have $m_{\beta_1, \beta_2}=0$. \\

(iii)
If $\beta_1 \neq \beta_2$, then
\begin{align*}
m_{\beta_1, \beta_2}=
\sum_{i, j} n_{0, \beta_1}(S_i) g^{ij} n_{0, \beta_2}(S_j)
+m_{\beta_1, \beta_2-\beta_1}+m_{\beta_1-\beta_2, \beta_2}.
\end{align*}

(iv)
If $\beta_1=\beta_2=\beta$, we have
\begin{align*}
m_{\beta, \beta}=n_{0, \beta}(c_2(X))
+\sum_{i, j}n_{0, \beta}(S_i)g^{ij} n_{0, \beta}(S_j)
-\sum_{\beta_1+\beta_2=\beta}
m_{\beta_1, \beta_2}.
\end{align*}
In~\cite{KP}, both of the invariants (\ref{g=0 GV}), (\ref{g=1 GV})
are conjectured to be integers, and Gromov-Witten invariants on $X$ are computed to support the conjectures in many examples
by localization technique or mirror symmetry. Note that the genus zero integrality conjecture has been proved 
by Ionel-Parker \cite[Theorem 9.2]{IP} using symplectic geometry.

\subsection{Review of $\DT_4$ invariants}\label{subsec:DT4}

We fix an ample divisor $\omega$ on $X$
and take a cohomology class
$v \in H^{\ast}(X, \mathbb{Q})$.
The coarse moduli space $M_{\omega}(v)$
of $\omega$-Gieseker semistable sheaves
$E$ on $X$ with $\ch(E)=v$ exists as a projective scheme.
We always assume that
$M_{\omega}(v)$ is a fine moduli space, i.e.~any point $[E] \in M_{\omega}(v)$ is stable and
there is a universal family
\begin{align*}
\eE \in \Coh(X \times M_{\omega}(v))
\end{align*}
flat over $M_\omega(v)$.
For instance, the moduli space of 1-dimensional stable sheaves $E$ with $[E]=\beta$, $\chi(E)=1$ and Hilbert schemes of 
closed subschemes satisfy this assumption \cite{CK1, CMT1}.

In~\cite{BJ, CL1}, under certain hypotheses,
the authors construct 
a virtual
class
\begin{align}\label{virtual}
[M_{\omega}(v)]^{\rm{vir}} \in H_{2-\chi(v, v)}(M_{\omega}(v), \mathbb{Z}), \end{align}
where $\chi(-,-)$ denotes the Euler pairing.
Notice that this class could a priori be non-algebraic.

Roughly speaking, in order to construct such a class, one chooses at
every point $[E]\in M_{\omega}(v)$, a half-dimensional real subspace
\begin{align*}\Ext_{+}^2(E, E)\subseteq \Ext^2(E, E)\end{align*}
of the usual obstruction space $\Ext^2(E, E)$, on which the quadratic form $Q$ defined by Serre duality is real and positive definite. 
Then one glues local Kuranishi-type models of the form 
\begin{equation}\kappa_{+}=\pi_+\circ\kappa: \Ext^{1}(E,E)\to \Ext_{+}^{2}(E,E),  \nonumber \end{equation}
where $\kappa$ is the Kuranishi map for $M_{\omega}(v)$ at $[E]$ and $\pi_+$ denotes projection 
on the first factor of the decomposition $\Ext^{2}(E,E)=\Ext_{+}^{2}(E,E)\oplus\sqrt{-1}\cdot\Ext_{+}^{2}(E,E)$.  

In \cite{CL1}, local models are glued in three special cases: 
\begin{enumerate}
\item when $M_{\omega}(v)$ consists of locally free sheaves only, 
\item  when $M_{\omega}(v)$ is smooth,
\item when $M_{\omega}(v)$ is a shifted cotangent bundle of a quasi-smooth derived scheme. 
\end{enumerate}
In each case, the corresponding virtual classes are constructed using either gauge theory or algebraic-geometric perfect obstruction theory.

The general gluing construction is due to Borisov-Joyce \cite{BJ}, 
based on Pantev-T\"{o}en-Vaqui\'{e}-Vezzosi's theory of shifted symplectic geometry \cite{PTVV} and Joyce's theory of derived $C^{\infty}$-geometry.
The corresponding virtual class is constructed using Joyce's
D-manifold theory (a machinery similar to Fukaya-Oh-Ohta-Ono's theory of Kuranishi space structures used for defining Lagrangian Floer theory).

To construct the above virtual class (\ref{virtual}) with coefficients in $\mathbb{Z}$ (instead of $\mathbb{Z}_2$), we need an orientability result 
for $M_{\omega}(v)$, which can be stated as follows.
Let  
\begin{equation*}
 \lL:=\mathrm{det}(\dR \hH om_{\pi_M}(\eE, \eE))
 \in \Pic(M_{\omega}(v)), \quad  
\pi_M \colon X \times M_{\omega}(v)\to M_{\omega}(v)
\end{equation*}
be the determinant line bundle of $M_{\omega}(v)$, which is equipped with the nondegenerate symmetric pairing $Q$ induced by Serre duality.  An \textit{orientation} of 
$(\mathcal{L},Q)$ is a reduction of its structure group from $\mathrm{O}(1,\mathbb{C})$ to $\mathrm{SO}(1, \mathbb{C})=\{1\}$. In other words, we require a choice of square root of the isomorphism
\begin{equation*}Q: \lL\otimes \lL \to \oO_{M_{\omega}(v)}.  \end{equation*}
Existence of orientations is first proved when the Calabi-Yau 4-fold $X$ satisfies 
$\mathrm{Hol}(X)=\mathrm{SU}(4)$ and $H^{\rm{odd}}(X,\mathbb{Z})=0$ \cite{CL2}, 
and is recently generalized to arbitrary Calabi-Yau 4-folds \cite{CGJ}.
Notice that the collection of orientations forms a torsor for $H^{0}(M_{\omega}(v),\mathbb{Z}_2)$. 

Examples computed in Section \ref{sect on examples} only involve virtual class constructions in (2) and (3) mentioned above.
We briefly review them (modulo discussions on choices of orientations) as follows:  
\begin{itemize}
\item When $M_{\omega}(v)$ is smooth, the obstruction sheaf $\mathrm{Ob}\to M_{\omega}(v)$ is a vector bundle endowed with a quadratic form $Q$ via Serre duality. Then the virtual class is given by
\begin{equation}[M_{\omega}(v)]^{\rm{vir}}=\mathrm{PD}(e(\mathrm{Ob},Q)).   \nonumber \end{equation}
Here $e(\mathrm{Ob}, Q)$ is the half-Euler class of 
$(\mathrm{Ob},Q)$ (i.e.~the Euler class of its real form $\mathrm{Ob}_+$, the orientation of $(\mathcal{L},Q)$ in this case is equivalent to the
orientation of $\mathrm{Ob}_+$), 
and $\mathrm{PD}(-)$ denotes 
Poincar\'e dual. 
Note that the half-Euler class satisfies 
\begin{align*}
e(\mathrm{Ob},Q)^{2}&=(-1)^{\frac{\mathrm{rk}(\mathrm{Ob})}{2}}e(\mathrm{Ob}),  \qquad \textrm{ }\mathrm{if}\textrm{ } \mathrm{rk}(\mathrm{Ob})\textrm{ } \mathrm{is}\textrm{ } \mathrm{even}, \\
 e(\mathrm{Ob},Q)&=0, \qquad\qquad\qquad\qquad \ \,  \textrm{ }\mathrm{if}\textrm{ } \mathrm{rk}(\mathrm{Ob})\textrm{ } \mathrm{is}\textrm{ } \mathrm{odd}. 
\end{align*}
\item Suppose $M_{\omega}(v)$ is the classical truncation of the shifted cotangent bundle of a quasi-smooth derived scheme. Roughly speaking, this means that at any closed point $[E]\in M_{\omega}(v)$, we have a Kuranishi map of the form
\begin{equation}\kappa \colon
 \Ext^{1}(E,E)\to \Ext^{2}(E,E)=V_E\oplus V_E^{*},  \nonumber \end{equation}
where $\kappa$ factors through a maximal isotropic subspace $V_E$ of $(\Ext^{2}(E,E),Q)$. Then the virtual class of $M_{\omega}(v)$ is, 
roughly speaking, the 
virtual class of the perfect obstruction theory formed by $\{V_E\}_{E\in M_{\omega}(v)}$. 
When $M_{\omega}(v)$ is furthermore smooth as a scheme, 
then it is
simply the Euler class of the vector bundle 
$\{V_E\}_{E\in M_{\omega}(v)}$ over $M_{\omega}(v)$. 
\end{itemize}

\subsection{Sheaf theoretical interpretation with primary insertions}

In \cite{CMT1}, a sheaf theoretical interpretation of genus zero GV type invariants \eqref{g=0 GV} is proposed using $\DT_4$ invariants.

To be precise, let $M_1(X, \beta)$ be the moduli scheme of 
one dimensional stable sheaves $F$ on $X$ with $[F]=\beta\in H_2(X,\mathbb{Z})$ and $\chi(F)=1$. 
There exists a virtual class (see \cite[Section 3.2]{CT19}):
$$[M_1(X,\beta)]^{\mathrm{vir}}\in H_2(M_1(X,\beta),\mathbb{Z}), $$
in the sense of Borisov-Joyce \cite{BJ}, which depends on the choice of an orientation \cite{CGJ}.
For integral classes $\gamma \in H^{4}(X, \mathbb{Z})$, consider insertions: 
\begin{align}\label{insertion ch3 for one dim sheaves} 
\tau_0 \colon H^{4}(X, \mathbb{Z})\to H^{2}(M_1(X,\beta), \mathbb{Z}), \quad 
\tau_0(\gamma):=(\pi_{M})_{\ast}(\pi_X^{\ast}\gamma \cup\ch_{3}(\mathbb{F}) ),
\end{align}
where $\pi_X$, $\pi_M$ are projections from $X \times M_1(X,\beta)$
to corresponding factors, $\mathbb{F}$ is a universal sheaf unique up to a twist of a line bundle from $M_1(X,\beta)$. 

Note that $\ch_3(\mathbb{F})$ is the
Poincar\'e dual to the fundamental cycle of $\mathbb{F}$ which is independent of the choice of $\mathbb{F}$.
We can define $\DT_4$ invariants with primary insertions on $M_1(X,\beta)$:
\begin{align*}\langle\tau_0(\gamma) \rangle_{\beta}:=\int_{[M_{1}(X,\beta)]^{\rm{vir}}}\tau_0(\gamma). \end{align*}
In \cite{CMT1}, this invariant is proposed to recover genus zero GV type invariants. 
\begin{conj}\emph{(\cite[Conjecture 0.2]{CMT1})}\label{g=0 one dim sheaf conj}
For a certain choice of an orientation on $M_1(X, \beta)$, we have 
$$\langle\tau_0(\gamma) \rangle_{\beta}=n_{0,\beta}(\gamma), $$
where $n_{0,\beta}(\gamma)$ is the g=0 Gopakumar-Vafa type invariant (\ref{g=0 GV}).
\end{conj}
In \cite{CMT2}, the authors also proposed how to recover all genus GV type invariants via stable pair theory which is a CY 4-fold analogy of Pandharipande-Thomas' work \cite{PT} on CY 3-folds. 
The motivation of this work is to give a sheaf theoretical interpretation of 
genus one invariants via moduli spaces of one dimensional stable sheaves, which is in spirit 
an analogy of the work of Gopakumar-Vafa \cite{GV} and Maulik-Toda \cite{MT}.

\subsection{$\DT_4$ descendent invariants and conjecture}
Of course, the theory of sheaves of vanishing cycles is infeasible on Calabi-Yau 4-folds. To capture genus one GV type invariants,
we propose to use the insertion \eqref{insertion ch3 for one dim sheaves} involving higher Chern characters ($\ch_4$). 

To fix the ambiguity in choosing a universal sheaf, we define 
\begin{align}\label{norm univ sheaf}\mathbb{F}_{\mathrm{norm}}:=\mathbb{F}\otimes\pi_M^*(\det(\pi_{M*}\mathbb{F}))^{-1}. \end{align}
It is independent of the choice of $\mathbb{F}$ and its push-forward to $M_1(X,\beta)$ has trivial determinant. 
For $\alpha\in H^2(X,\mathbb{Z})$, we define (descendent) insertions 
\begin{align*}  
\tau_1 \colon H^{2}(X, \mathbb{Z})\to H^{2}(M_1(X,\beta), \mathbb{Q}), \quad 
\tau_1(\alpha):=(\pi_{M})_{\ast}(\pi_X^{\ast}\alpha \cup\ch_{4}(\mathbb{F}_{\mathrm{norm}}) ),
\end{align*}
and $\DT_4$ invariants with descendent insertions:
\begin{align}\label{des invs} 
\langle\tau_1(\alpha) \rangle_{\beta}:=\int_{[M_{1}(X,\beta)]^{\rm{vir}}}\tau_1(\alpha). \end{align}
\begin{conj}\label{main conj}
For a certain choice of an orientation on $M_1(X, \beta)$, there is an equality of functions of $\alpha \in H^2(X)$ 
\begin{align*}\langle\tau_1(\alpha) \rangle_{\beta}=\frac{n_{0,\beta}(\alpha^2)}{2\,(\alpha\cdot\beta)}
-\sum_{\beta_1+\beta_2=\beta}\frac{(\alpha\cdot\beta_1)(\alpha\cdot\beta_2)}{4\,(\alpha\cdot\beta)}m_{\beta_1,\beta_2}
-\sum_{k\geqslant1,\,k|\beta}\frac{(\alpha\cdot\beta)}{k}\,n_{1,\beta/k}.  \end{align*}
Here $n_{0,\beta}(-)$ and $n_{1,\beta}$ are genus 0 and 1 GV type invariants \eqref{g=0 GV}, \eqref{g=1 GV} and $m_{\beta_1,\beta_2}$ are meeting invariants \eqref{meeting invs} 
which can be inductively determined by genus 0 GV type invariants.
\end{conj}
We believe the choices of orientations in Conjecture \ref{g=0 one dim sheaf conj} and Conjecture \ref{main conj} are the same when moduli spaces are connected. This will be checked in examples considered in Section \ref{sect on examples}.

Combining Conjecture \ref{g=0 one dim sheaf conj} and Conjecture \ref{main conj},  
all genus GV type invariants can be interpreted in terms of $\DT_4$ counting invariants of one dimensional stable sheaves.

\subsection{Heuristic argument}\label{heuristic argue}
In this section, we justify Conjecture \ref{main conj} in an ideal situation: let $X$ be an `ideal' $\mathrm{CY_{4}}$
in the sense that all curves of $X$ deform in families of expected dimensions, and have expected generic properties, i.e.
\begin{enumerate}
\item
any rational curve in $X$ is a chain of smooth $\mathbb{P}^1$'s with normal bundle $\mathcal{O}_{\mathbb{P}^{1}}(-1,-1,0)$, and
moves in a compact 1-dimensional smooth family of embedded rational curves, whose general member is smooth with 
normal bundle $\mathcal{O}_{\mathbb{P}^{1}}(-1,-1,0)$. 
\item
any elliptic curve $E$ in $X$ is smooth, super-rigid, i.e. 
the normal bundle is 
$L_1 \oplus L_2 \oplus L_3$
for general degree zero line bundle $L_i$ on $E$
satisfying $L_1 \otimes L_2 \otimes L_3=\oO_E$. 
Furthermore any two elliptic curves are 
disjoint and disjoint from all families of rational curves on $X$. 
\item
there is no curve in $X$ with genus $g\geqslant 2$.
\end{enumerate}
For a fixed curve class $\beta$, let $D_{\beta}$ be the total space of one dimensional family of rational curves in 
class $\beta$. It has a fibration 
$$\pi \colon D_\beta \to M_\beta $$
to the (one dimensional) moduli space $M_\beta$ of rational curves in class $\beta$. In the ideal situation, most fibers 
of $\pi$ are smooth $\mathbb{P}^1$ with normal bundle $\mathcal{O}_{\mathbb{P}^{1}}(-1,-1,0)$ 
except finitely many reducible fibers. As in \cite[Sect. 2.3]{KP},
 we may assume 
$D_\beta$ is modelled on the blow-up of a $\mathbb{P}^1$-bundle over finitely many general points.
We first determine its contribution to our $\DT_4$ descendent invariant.
The universal family of $M_1(X,\beta)$ is $j_*\oO_{D_\beta}$ for 
the embedding 
\begin{align*}
j \colon 
D_\beta \hookrightarrow X\times M_{\beta}.
\end{align*}
It is obvious that its push-forward to the moduli space has trivial determinant.
By the Grothendieck-Riemann-Roch formula, we have 
$$\ch_4(j_{\ast}\oO_{D_\beta})=-\frac{1}{2}c_1(\omega_{\pi}). $$
Therefore for the family of rational curves (with respect to the orientation in \cite{CMT1}, i.e. taking the usual fundamental class 
of the moduli space), we have 
$$\int_{M_{\beta}}\tau_1(\alpha)=-\frac{1}{2}\int_{D_\beta}c_1(\omega_{\pi})\cdot \alpha|_{D_\beta}. $$
Below we explicitly compute this intersection number.
Suppose that $D_{\beta}$ is the blow-up 
of general $k$-points of a $\mathbb{P}^1$-bundle.
Then 
$$H^2(D_\beta,\mathbb{Q})=\mathbb{Q}\langle \beta,\psi,\beta_1^{(1)},
\ldots, \beta_1^{(k)}\rangle
$$ 
is generated by $k+2$ elements, where $\beta$ is the fiber class, 
$\beta_1^{(i)}$ are the classes of exceptional curves, 
$\psi=c_1(\omega_{\pi})$ is the first Chern class of the relative cotangent bundle of $\pi$. 

We have the following intersection numbers
\begin{align*}
\beta\cdot \beta=\beta_1^{(i)} \cdot \beta=0,\quad \beta\cdot \psi=-2,\quad \beta_1^{(i)}\cdot\psi=(\beta_1^{(i)})^2=-1. 
\end{align*}
The restriction of $\alpha$ to $D_{\beta}$ satisfies  
\begin{align*}
\alpha|_{D_{\beta}}=a \beta+b\psi+
\sum_{i=1}^k d_i \beta_1^{(i)},
\end{align*}
for some $a, b, d_i \in \mathbb{Q}$. 
We have the following equations
\begin{align}
\label{eqn1}
&(\alpha|_{D_{\beta}})^2=b^2\psi^2-\sum_{i=1}^k d_i^2-4ab-2b \sum_{i=1}^k d_i, \\
\label{eqn2}
&-\frac{1}{2}\psi\cdot \alpha|_{D_{\beta}}=a-\frac{b}{2}\psi^2+
\frac{1}{2} \sum_{i=1}^k d_i, \\
\label{eqn3}
& \alpha \cdot \beta=-2b, \quad  \alpha \cdot \beta_1^{(i)}=-b-d_i. 
\end{align}
We set $\beta_2^{(i)}=\beta-\beta_1^{(i)}$. 
From equation (\ref{eqn3}), we 
can solve $b, d_i$ and obtain
\begin{align*}
b=-\frac{1}{2} \alpha \cdot \beta, \quad d_i=\frac{1}{2} \alpha \cdot (\beta_2^{(i)}-\beta_1^{(i)}).
\end{align*}
By substituting into (\ref{eqn2}), 
we can solve $a$ and obtain
\begin{align*}
a=-\frac{1}{2} \psi \cdot \alpha|_{D_{\beta}}-\frac{\psi^2}{4}
(\alpha \cdot \beta)+\frac{1}{4}\sum_{i=1}^k \alpha \cdot (\beta_2^{(i)}-\beta_1^{(i)}). 
\end{align*}
By substituting these 
into (\ref{eqn1}),
we obtain 
\begin{align*}
-\frac{1}{2}\psi\cdot \alpha|_{D_{\beta}}=\frac{(\alpha|_{D_{\beta}})^2}{2(\alpha\cdot \beta)}+\frac{(\alpha\cdot \beta)}{8}\psi^2+
\sum_{i=1}^k \frac{(\alpha \cdot \beta_1^{(i)}-\alpha \cdot \beta_2^{(i)})^2}{8(\alpha \cdot \beta)}.
\end{align*}
However some $\beta_1^{(i)}$ may be equal to $\beta_1^{(j)}$ in the 
ambient space $X$, and the number of times this could happen is given by the meeting invariant $m_{\beta_1,\beta_2}$ 
for such $\beta_1=\beta_1^{(i)}$ and $\beta_2=\beta_2^{(i)}$
(ref. \cite[pp. 10]{KP}). 
We also need to 
divide it by two in order to modulo the symmetry between $\beta_1$ and $\beta_2=\beta-\beta_1$. Therefore, the above formula should be adjusted to be
$$-\frac{1}{2}\psi\cdot \alpha|_{D_{\beta}}=\frac{(\alpha|_{D_{\beta}})^2}{2(\alpha\cdot \beta)}+\frac{(\alpha\cdot \beta)}{8}\psi^2+
\sum_{\beta_1+\beta_2=\beta}\frac{(\alpha\cdot \beta_1-\alpha\cdot \beta_2)^2}{8(\alpha\cdot \beta)}\cdot \frac{m_{\beta_1,\beta_2}}{2}.$$
From \cite[pp. 11]{KP}, we have 
$$\psi^2=-\frac{1}{2}\sum_{\beta_1+\beta_2=\beta}m_{\beta_1,\beta_2}. $$
Therefore we obtain  
$$-\frac{1}{2}\psi\cdot \alpha|_{D_{\beta}}=\frac{(\alpha|_{D_{\beta}})^2}{2(\alpha\cdot \beta)}-\sum_{\beta_1+\beta_2=\beta}\frac{(\alpha\cdot\beta_1)(\alpha\cdot\beta_2)}{4\,(\alpha\cdot\beta)}m_{\beta_1,\beta_2}.$$
Thus considering only contribution from the rational curves family, we have 
\begin{align*}\langle\tau_1(\alpha) \rangle_{\beta}=\frac{n_{0,\beta}(\alpha^2)}{2\,(\alpha\cdot\beta)}
-\sum_{\beta_1+\beta_2=\beta}\frac{(\alpha\cdot\beta_1)(\alpha\cdot\beta_2)}{4\,(\alpha\cdot\beta)}m_{\beta_1,\beta_2}. 
\end{align*}
Next, we determine the contribution from a super-rigid elliptic curve $E$.  
The moduli space $M_1(X,\beta)$ of one dimensional stable sheaves which are supported on $E$ with $[E]=\beta$ is $\Pic^1(E)$, 
which is identified with $E$ by 
the map $x \mapsto \oO_E(x)$ for $x \in E$.  
It has a normalized universal sheaf  
\begin{align*}
\mathbb{F}=j_{\ast}\oO_{E\times E}(\Delta) \in \Coh(E \times E), 
\end{align*}
where $\Delta\subset E\times E$ is the diagonal, 
$j \colon E\times E\to X\times E$ is the product of the natural inclusion $E\hookrightarrow X$ and 
the identity map on $E$. 
For $\alpha \in H^2(X)$, a 
direct calculation gives 
\begin{align*}
\int_{\Pic^1(E)}\tau_1(\alpha)=\alpha\cdot [E]=\alpha\cdot \beta. 
\end{align*}
As for $r\geqslant1$ and $r|\beta$ (i.e. $\beta/r\in H_2(X,\mathbb{Z})$) and a super-rigid elliptic curve $E$ with $[E]=\beta/r$, 
the moduli space of one dimensional stable sheaves supported on $E$ is the moduli space $M_{E}(r,1)$ of rank $r$ degree one stable bundles 
on $E$. The determinant map defines an isomorphism 
\begin{align*}
\det \colon M_{E}(r,1)\stackrel{\cong}{\to}  M_{E}(1,1)=\Pic^1(E) \stackrel{\cong}{\leftarrow} E.
\end{align*}
Under the above isomorphism, 
the normalized universal sheaf $\mathbb{F}$ has the following $K$-theory 
class
\begin{align*}
[\mathbb{F}]=[j_{\ast}\oO_{E\times E}(\Delta)]+(r-1)[j_{\ast}\oO_{E \times E}].
\end{align*}
For example, the above identity follows from a simpler version of the argument 
in Lemma~\ref{lem:Fnorm}.
Therefore we similarly get 
\begin{align*}
\int_{M_{E}(r,1)}\tau_1(\alpha)=\frac{\alpha\cdot \beta}{r}. 
\end{align*}
Therefore all super-rigid elliptic curves contribute 
\begin{align*}
\sum_{r|\beta}\frac{(\alpha\cdot \beta)}{r}\cdot n_{1,\beta/r} 
\end{align*}
to our descendent invariant $\langle\tau_1(\alpha) \rangle_{\beta}$ for some choice of orientations.
Here we did not explain why minus sign in front of genus one GV type invariants comes as the right choice of orientations. In fact, 
this choice emerges from the experimental study of all examples in Section \ref{sect on examples}. It will be an interesting 
question to give an explanation of this choice in the above ideal geometry.


\subsection{Constraints on genus zero GV type invariants}
In order to make sense of the formula 
in Conjecture~\ref{main conj}, 
the right hand side has to be a linear function on $\alpha$.
A priori, the right hand side is a rational function with 
simple pole along the hyperplane $\alpha \cdot \beta=0$, 
so we need to show the absence of pole of this 
rational function. 
It requires the following constraint 
\begin{align}\label{constraint}
n_{0,\beta}(\alpha^2)=\frac{1}{2}\sum_{\begin{subarray}{c}\beta_1+\beta_2=\beta  \\ \beta_1, \beta_2>0 \end{subarray}}(\alpha\cdot\beta_1)(\alpha\cdot\beta_2)\, m_{\beta_1,\beta_2},
\end{align}
for any $\alpha \in H^2(X)$ with $\alpha \cdot \beta=0$. 
Below we show that the above formula can be derived from 
the WDVV equation. 
We first solve the recursion for the meeting numbers and 
express the identity (\ref{constraint}) in terms of
genus zero GV type invariants. 
\begin{lem}\label{lem:GV0}
The identity (\ref{constraint}) is equivalent to the 
identity
\begin{align}\label{constraint:equiv}
n_{0, \beta}(\alpha^2)=\frac{1}{2} 
\sum_{\begin{subarray}{c}k_1,k_2 \in \mathbb{Z}_{>0} \\
\mathrm{g.c.d.}(k_1, k_2)=1\end{subarray}
}\sum_{\begin{subarray}{c}
\beta_1+\beta_2=\beta \\
k_1|\beta_1, k_2|\beta_2 \end{subarray}}
\sum_{a, b}\frac{(\alpha \cdot \beta_1)(\alpha \cdot \beta_2)}{k_1^2 k_2^2}
n_{0, \beta_1/k_1}(S_a)g^{ab}n_{0, \beta_2/k_2}(S_b). 
\end{align}
Here $\mathrm{g.c.d.}(k_1, k_2)$ denotes the greatest common divisor of $k_1$ and $k_2$, $\{S_*\}$ is a basis of the free part of $H^4(X,\mathbb{Z})$
and $(g^{ab})$ is the inverse matrix of $(g_{ab})$
given in \eqref{kunneth decom}. 
\end{lem}
\begin{proof}
For each $k_1, k_2$ with $\mathrm{g.c.d.}(k_1, k_2)=1$, we set 
\begin{align*}
a_{k_1, k_2}&:=\sum_{\begin{subarray}{c}
\beta_1+\beta_2=\beta \\
k_1|\beta_1, k_2|\beta_2 \end{subarray}}
\frac{1}{k_1^2 k_2^2}(\alpha \cdot \beta_1)(\alpha \cdot \beta_2)\,m_{\beta_1/k_1, \beta_2/k_2}, \\
b_{k_1, k_2}&:=\sum_{\begin{subarray}{c}
\beta_1+\beta_2=\beta \\
k_1|\beta_1, k_2|\beta_2 \end{subarray}}
\sum_{a, b}
\frac{1}{k_1^2 k_2^2}(\alpha \cdot \beta_1)(\alpha \cdot \beta_2)\,n_{0, \beta_1/k_1}(S_a)g^{ab}n_{0, \beta_2/k_2}(S_b).
\end{align*}
Since $\alpha \cdot \beta=0$, the only terms of
$\beta_1, \beta_2$ which are non-proportional contribute 
to the above sums. 
Therefore by the recursion for the meeting numbers \eqref{meeting invs}, 
we obtain 
\begin{align}\notag
a_{k_1, k_2}-b_{k_1, k_2}
&=\sum_{\begin{subarray}{c}
\beta_1+\beta_2=\beta \\
k_1|\beta_1, k_2|\beta_2 \end{subarray}}
\frac{1}{k_1^2 k_2^2}(\alpha \cdot \beta_1)(\alpha \cdot \beta_2)\,m_{\beta_1/k_1-\beta_2/k_2, \beta_2/k_2} \\
&+\sum_{\begin{subarray}{c}
\beta_1+\beta_2=\beta \\
\label{ab}
k_1|\beta_1, k_2|\beta_2 \end{subarray}}
\frac{1}{k_1^2 k_2^2}(\alpha \cdot \beta_1)(\alpha \cdot \beta_2)\,m_{\beta_1/k_1, \beta_2/k_2-\beta_1/k_1}.
\end{align}
In the above first sum, 
we set 
\begin{align*}
\beta_1':=k_1 \left(\frac{\beta_1}{k_1}- \frac{\beta_2}{k_2} \right), \quad 
\beta_2':=(k_1+k_2) \frac{\beta_2}{k_2}.
\end{align*}
Then 
we have 
$k_1|\beta_1'$, $(k_1+k_2)|\beta_2'$, and 
\begin{align*}
\beta_1'+\beta_2'=\beta, \quad 
m_{\beta_1/k_1-\beta_2/k_2, \beta_2/k_2}=m_{\beta_1'/k_1, \beta_2'/(k_1+k_2)}.
\end{align*}
Conversely, suppose that $\beta_1'+\beta_2'=\beta$
is given such that $k_1|\beta_1'$ and $(k_1+k_2)|\beta_2'$.
Then we have $\beta_1+\beta_2=\beta$, 
where $\beta_1$, $\beta_2$ are given by   
\begin{align}\label{rel:beta}
\beta_1=k_1\left(\frac{\beta_1'}{k_1}+
\frac{\beta_2'}{k_1+k_2}\right), \quad 
\beta_2=k_2 \frac{\beta_2'}{k_1+k_2},
\end{align}
which satisfy $k_1|\beta_1$, $k_2|\beta_2$.
Therefore we can write the first sum in (\ref{ab})
as a sum over $\beta_1'+\beta_2'=\beta$
with $k_1|\beta_1'$ and $(k_1+k_2)|\beta_2'$. 
Moreover using the relation $\alpha \cdot \beta=0$
and (\ref{rel:beta}), we have 
\begin{align*}
\frac{1}{k_1^2 k_2^2}(\alpha \cdot \beta_1)(\alpha \cdot \beta_2)
=\frac{1}{k_1^2 (k_1+k_2)^2}(\alpha \cdot \beta_1')(\alpha \cdot \beta_2'). 
\end{align*}
Therefore the first sum of (\ref{ab})
is 
\begin{align*}
\sum_{\begin{subarray}{c}
\beta_1'+\beta_2'=\beta \\
k_1|\beta_1', (k_1+k_2)|\beta_2' \end{subarray}}
\frac{1}{k_1^2 (k_1+k_2)^2}(\alpha \cdot \beta_1')(\alpha \cdot \beta_2')\,m_{\beta_1'/k_1, \beta_2'/(k_1+k_2)}
=a_{k_1, k_1+k_2}. 
\end{align*}
Similarly the second sum in (\ref{ab}) 
equals to $a_{k_1+k_2, k_2}$. 
Therefore we obtain the relation
\begin{align*}
a_{k_1, k_2}=b_{k_1, k_2}+a_{k_1, k_1+k_2}+a_{k_1+k_2, k_2}. 
\end{align*} 
Since the right hand side of (\ref{constraint}) 
is $a_{1, 1}/2$, 
the identity (\ref{constraint}) is equivalent to 
\begin{align*}
n_{0, \beta}(\alpha^2)=\frac{1}{2} 
\sum_{k_1,k_2 \in \mathbb{Z}_{>0} }
b_{k_1, k_2}.
\end{align*}
Here in the above sum, $(k_1, k_2)$ are those 
which appear by iterated operations 
of either $(k_1, k_2) \mapsto (k_1, k_1+k_2)$
or $(k_1, k_2) \mapsto (k_1+k_2, k_2)$. 
By an elementary Euclidean algorithm
argument, every coprime pair $(k_1, k_2)$ 
appears exactly once by the above operations. 
Therefore the identity (\ref{constraint}) is equivalent to 
(\ref{constraint:equiv}). 
\end{proof}

By the above lemma, 
we can give an equivalent form of (\ref{constraint}) in terms of 
genus zero Gromov-Witten invariants. 
\begin{prop}\label{prop:GW}
The identity (\ref{constraint}) for all $(\alpha, \beta)$
with $\alpha \cdot \beta=0$ is equivalent to 
the following identity for all $(\alpha, \beta)$
with $\alpha \cdot \beta=0$,
\begin{align}\label{constraint:2}
\mathrm{GW}_{0, \beta}(\alpha^2)
=\frac{1}{2}\sum_{\beta_1+\beta_2=\beta}\sum_{a, b}
(\alpha \cdot \beta_1)(\alpha \cdot \beta_2)\,\mathrm{GW}_{0, \beta_1}(S_a)g^{ab}\mathrm{GW}_{0, \beta_2}(S_b), 
\end{align}
where $\{S_*\}$ is a basis of the free part of $H^4(X,\mathbb{Z})$.
\end{prop}
\begin{proof}
Suppose that (\ref{constraint}) holds for all $(\alpha, \beta)$
with $\alpha \cdot \beta=0$.
Then by Lemma~\ref{lem:GV0}, we have 
\begin{align*}
\mathrm{GW}_{0, \beta}(\alpha^2)
&=
\sum_{k\geqslant 1, k|\beta}\frac{1}{k^2} n_{0, \beta/k}(\alpha^2) \\
&=\frac{1}{2}\sum_{k\geqslant 1, k|\beta}
\sum_{\begin{subarray}{c}k_1 \in \mathbb{Z}_{>0}, k_2 \in \mathbb{Z}_{>0} \\
\mathrm{g.c.d.}(k_1, k_2)=1\end{subarray}
}\sum_{\begin{subarray}{c}
\beta_1+\beta_2=\beta/k \\
k_1|\beta_1, k_2|\beta_2 \end{subarray}}
\sum_{a, b}\frac{1}{k^2k_1^2 k_2^2}(\alpha \cdot \beta_1)(\alpha \cdot \beta_2)
n_{0, \beta_1/k_1}(S_a)g^{ab}n_{0, \beta_2/k_2}(S_b) \\
&=\frac{1}{2}
\sum_{k_1' \in \mathbb{Z}_{>0}, k_2' \in \mathbb{Z}_{>0}}
\sum_{\begin{subarray}{c}
\beta_1'+\beta_2'=\beta \\
k_1'|\beta_1', k_2'|\beta_2' \end{subarray}}
\sum_{a, b}\frac{1}{k_1^{'2} k_2^{'2}}(\alpha \cdot \beta_1')(\alpha \cdot \beta_2')
n_{0, \beta_1'/k_1'}(S_a)g^{ab}n_{0, \beta_2'/k_2'}(S_b).
 \end{align*}
Here we have set $\beta_1'=k\beta_1$, $\beta_2'=k\beta_2'$
and $k_1'=k k_1$, $k_2'=k k_2$ in the third identity, 
where $\mathrm{g.c.d.}(k_1', k_2')=k$ so that there is no constraint 
on $\mathrm{g.c.d.}(k_1', k_2')$ in the third identity. 
By the multiple cover formula, 
the last identity gives the desired identity (\ref{constraint:2}).

Conversely suppose (\ref{constraint:2}) holds for all 
$(\alpha, \beta)$ with $\alpha \cdot \beta=0$. 
Then by applying the above computation backwards, 
we can conclude (\ref{constraint}) 
by an induction on the divisibility of $\beta$.  
\end{proof}
\begin{thm}\label{constraint holds}
The identity \eqref{constraint} holds for any Calabi-Yau 4-fold, i.e. the formula in Conjecture \ref{main conj} is 
an identity of linear functions on $\alpha\in H^2(X)$.
\end{thm}
\begin{proof}
We assume $H^{\mathrm{odd}}(X,\mathbb{Q})=0$ for simplicity and the same argument works in the general case.
Let $S_0=1,S_1,\cdots,S_{m-1},S_{m}=[\mathrm{pt}]$ be a basis of $H^*(X,\mathbb{Q})$, where $S_1,\cdots, S_r\in H^2(X,\mathbb{Q})$, 
$S_{r+1},\cdots, S_l\in H^4(X,\mathbb{Q})$ and $S_{l+1},\cdots, S_{m-1}\in H^6(X,\mathbb{Q})$.
Write $\gamma:=\sum_{i=0}^m t_iS_i$ where $t_i$ are formal variables.
The Gromov-Witten potential (ref. \cite[(8.29)]{CoxKatz}) is 
\begin{align*}
\Phi&:=\sum_{n=0}^{\infty}\sum_{\beta\in H_2(X,\mathbb{Z})}\frac{1}{n!}\mathrm{GW}_{0, \beta}(\gamma^{\otimes n})\,q^{\beta} \\
&=\frac{1}{6}\int_X\gamma^3+\sum_{b\geqslant 0}\sum_{0\neq\beta\in H_2(X,\mathbb{Z})}\frac{1}{b!}\mathrm{GW}_{0,\beta}\left(\left(\sum_{i=r+1}^m t_iS_i\right)^{\otimes b}\right)e^{\int_\beta\left(\sum_{i=1}^r t_iS_i\right)}\,q^{\beta} \\
&=\frac{1}{6}\int_X\gamma^3+\sum_{i=r+1}^{l}t_i
\sum_{0\neq\beta\in H_2(X,\mathbb{Z})}\mathrm{GW}_{0,\beta}(S_i)e^{\int_\beta\left(\sum_{i=1}^r t_iS_i\right)}\,q^{\beta},
\end{align*}
where we use CY4 condition in the last equality. 

The WDVV equation (ref. \cite{W,DVV}, \cite[Theorem 8.2.4]{CoxKatz}) is 
$$\sum_{a,b}\Phi_{ija}g^{ab}\Phi_{bkl}=\sum_{a,b}\Phi_{jka}g^{ab}\Phi_{bil}, $$
where $\Phi_{abc}=\frac{\partial\Phi}{\partial t_a\partial t_b\partial t_c}$ denotes the third partial derivative of $\Phi$ with respect to $t_a, t_b, t_c$.

Note that
\begin{align*}
&\Phi_{11a}=\int_XS_1^2S_a+\sum_{0\neq\beta\in H_2(X,\mathbb{Z})}\mathrm{GW}_{0,\beta}(S_a)(S_1\cdot\beta)^2e^{\int_\beta(\sum_{i=1}^r t_iS_i)}\,q^{\beta}, \\
&\Phi_{22b}=\int_XS_2^2S_b+\sum_{0\neq\beta\in H_2(X,\mathbb{Z})}\mathrm{GW}_{0,\beta}(S_b)(S_2\cdot\beta)^2e^{\int_\beta(\sum_{i=1}^r t_iS_i)}\,q^{\beta}, \\
&\Phi_{12a}=\int_XS_1S_2S_a+\sum_{0\neq\beta\in H_2(X,\mathbb{Z})}\mathrm{GW}_{0,\beta}(S_a)(S_1\cdot\beta)(S_2\cdot\beta)e^{\int_\beta(\sum_{i=1}^r t_iS_i)}\,q^{\beta}. 
\end{align*}
By taking $i=j=1$, $k=l=2$ in the WDVV equation, we obtain  
\begin{align*}&\quad \mathrm{GW}_{0,\beta}\Big(\big((S_1\cdot\beta)S_2-(S_2\cdot\beta)S_1\big)^2\Big) \\
&=\sum_{r+1\leqslant a,b\leqslant l}\sum_{\beta_1+\beta_2=\beta}\mathrm{GW}_{0,\beta_1}(S_a)g^{ab}\mathrm{GW}_{0,\beta_2}(S_b)(S_1\cdot\beta_1)(S_2\cdot\beta_1)(S_1\cdot\beta_2)(S_2\cdot\beta_2) \\
&\quad  -\sum_{r+1\leqslant a,b\leqslant l}\sum_{\beta_1+\beta_2=\beta}\mathrm{GW}_{0,\beta_1}(S_a)g^{ab}\mathrm{GW}_{0,\beta_2}(S_b)(S_1\cdot\beta_1)^2(S_2\cdot\beta_2)^2 \\
&=\sum_{r+1\leqslant a,b\leqslant l}\sum_{\beta_1+\beta_2=\beta}\mathrm{GW}_{0,\beta_1}(S_a)g^{ab}\mathrm{GW}_{0,\beta_2}(S_b)(S_1\cdot\beta_1)(S_2\cdot\beta_1)(S_1\cdot\beta_2)(S_2\cdot\beta_2) \\
&\quad  -\sum_{r+1\leqslant a,b\leqslant l}\sum_{\beta_1+\beta_2=\beta}\mathrm{GW}_{0,\beta_1}(S_a)g^{ab}\mathrm{GW}_{0,\beta_2}(S_b)(S_1\cdot\beta_2)^2(S_2\cdot\beta_1)^2.
\end{align*}
Let $\alpha:=(S_1\cdot\beta)S_2-(S_2\cdot\beta)S_1\in H^2(X)$. By taking sum of the above two equations, we obtain
\begin{align*}\quad \mathrm{GW}_{0,\beta}(\alpha^2)=\frac{1}{2}\sum_{\beta_1+\beta_2=\beta}\sum_{a,b}(\alpha\cdot\beta_1)(\alpha\cdot\beta_2)\mathrm{GW}_{0,\beta_1}(S_a)g^{ab}\mathrm{GW}_{0,\beta_2}(S_b). \end{align*}
By considering other $i,j,k=l\in\{1,2,\cdots,r\}$, we can conclude for any $\alpha\in H^2(X)$ with $\alpha\cdot \beta=0$, the above equality holds.
Combining with Proposition \ref{prop:GW}, we are done.
\end{proof}
The following example shows that \eqref{constraint} gives nontrivial constraints on GW invariants of Calabi-Yau 4-folds whose Picard numbers are bigger than one.
\begin{exam}\label{exam on constrain}
When $X=\mathrm{Tot}_{\mathbb{P}^1\times \mathbb{P}^1}(\oO(-1,-1)\oplus \oO(-1,-1))$. The formula \eqref{constraint} (equivalently \eqref{constraint:2} by Proposition \ref{prop:GW}) gives a recursion formula for genus zero GW invariants of $X$:
\begin{align*}
\mathrm{GW}_{0,(d_1,d_2)}([\mathrm{pt}])=\frac{1}{2}\sum_{\begin{subarray}{c}d_1'+d_1''=d_1\\  d_2'+d_2''=d_2 \end{subarray}}
\frac{(d_1'd_2''-d_2'd_1'')^2}{d_1d_2}\mathrm{GW}_{0,(d_1',d_2')}([\mathrm{pt}])
\cdot \mathrm{GW}_{0,(d_1'',d_2'')}([\mathrm{pt}]).
\end{align*}
In \cite[Proposition 3, pp. 23]{KP}, GW invariants are explicitly computed by torus localization: 
\begin{align*}
\mathrm{GW}_{0,(d_1,d_2)}([\mathrm{pt}])=\frac{1}{(d_1+d_2)^2}
\dbinom{d_1+d_2}{d_1}^2.
\end{align*}
Combining them, we obtain a nontrivial identity 
of binomial coefficients for any $d_1,d_2\in \mathbb{Z}_{\geqslant 0}$,
\begin{align*}\frac{2d_1d_2}{(d_1+d_2)^2}
\dbinom{d_1+d_2}{d_1}^2=
\sum_{\begin{subarray}{c}d_i'+d_i''=d_i\\  d_i',d_i''>0, \,\, i=1,2 \end{subarray}}
\frac{(d_1'd_2''-d_2'd_1'')^2}{(d_1'+d_2')^2(d_1''+d_2'')^2}
\dbinom{d_1'+d_2'}{d_1'}
\dbinom{d_1''+d_2''}{d_1''}.
\end{align*}
\end{exam}

\subsection{Other $\DT_4$ descendent invariants}
In Conjecture \ref{g=0 one dim sheaf conj}, \ref{main conj}, we use $\DT_4$ invariants with $\ch_3$ and $\ch_4$ insertions to give 
sheaf theoretical interpretations of all genus GV type invariants of Calabi-Yau 4-folds. 
One can also consider other descendent invariants, though they do not seem to contain more information. 

For instance, with the normalized universal sheaf $\mathbb{F}_{\mathrm{norm}}$ \eqref{norm univ sheaf}, 
we can consider the $\ch_5$-descendent insertions
\begin{align*}
\tau_2 \colon H^{0}(X, \mathbb{Z})\to H^{2}(M_1(X,\beta), \mathbb{Q}), \quad 
\tau_2(r):=(\pi_{M})_{\ast}(\pi_X^{\ast}r \cup\ch_{5}(\mathbb{F}_{\mathrm{norm}}) ),
\end{align*}
and their $\DT_4$ invariants
\begin{align*} 
\langle\tau_2(r) \rangle_{\beta}:=\int_{[M_{1}(X,\beta)]^{\rm{vir}}}\tau_2(r). \end{align*}
Using an argument in the ideal geometry (as in Section \ref{heuristic argue}), we obtain the following conjectural formula
\begin{align*}
\langle\tau_2(1) \rangle_{\beta}=-\frac{1}{12}n_{0,\beta}(c_2(X)),
\end{align*}
for the same choice of orientations as in Conjecture \ref{g=0 one dim sheaf conj}. One can easily 
check the formula holds in examples considered in Section \ref{sect on examples} and 
the corresponding choice of orientations is unique and coincides with those in Conjecture \ref{g=0 one dim sheaf conj}, \ref{main conj}.

\section{Computations of Examples}\label{sect on examples}

\subsection{Elliptic fibrations}\label{sect on elliptic fib}
For $Y=\mathbb{P}^3$, we take general elements
\begin{align*}
u \in H^0(Y, \oO_Y(-4K_Y)), \quad 
v \in H^0(Y, \oO_Y(-6K_Y)).
\end{align*}
Let $X$
be a Calabi-Yau 4-fold with an elliptic fibration
\begin{align}\label{elliptic CY4}
\pi \colon X \to Y
\end{align}
given by the equation
\begin{align*}
zy^2=x^3 +uxz^2+vz^3
\end{align*}
in the $\mathbb{P}^2$-bundle
\begin{align*}
\mathbb{P}(\oO_Y(-2K_Y) \oplus \oO_Y(-3K_Y) \oplus \oO_Y) \to Y,
\end{align*}
where $[x:y:z]$ are homogeneous coordinates for the above
projective bundle. A general fiber of
$\pi$ is a smooth elliptic curve, and any singular
fiber is either a nodal or cuspidal plane curve.
Moreover, $\pi$ admits a section $\iota$ whose image
corresponds to the fiber point $[0: 1: 0]$.

Let $h$ be a hyperplane in $\mathbb{P}^3$, $f$ be a general fiber of $\pi \colon X\rightarrow Y$ and set
\begin{align}\label{div:BE}
B=:\pi^{\ast}h, \ D:=\iota(\mathbb{P}^3)\in H_{6}(X,\mathbb{Z}).
\end{align}
We consider moduli spaces $M_{1}(X, r[f])$ of 
one dimensional stable sheaves on $X$ in the multiple fiber classes $r[f]$ ($r\geqslant 1$).
\begin{lem}$($\cite[Lemma 2.1]{CMT1}$)$\label{lem:fib:isom}
For any $r \in \mathbb{Z}_{\geqslant 1}$, there is an isomorphism
$$M_{1}(X,r[f]) \cong X, $$
under which the virtual
class of $M_{1}(X,r[f])$ is given by
\begin{align}\label{Mvir}
[M_{1}(X,r[f])]^{\rm{vir}}=\pm \mathrm{PD}(c_3(X)) \in H_2(X, \mathbb{Z}),
\end{align}
where the sign corresponds to a choice of orientations in defining the LHS.
\end{lem}
By the isomorphism in Lemma~\ref{lem:fib:isom}, 
we have the equivalence
\begin{align*}
\Coh(X \times M_1(X, r[f])) \stackrel{\sim}{\to}
\Coh(X \times X). 
\end{align*}
Under the above equivalence, the 
normalized universal sheaf 
$\mathbb{F}_{\rm{norm}}$ 
is a coherent sheaf on $X \times X$. 
In order to compute $\tau_1(\alpha)_{r[f]}$, 
we need to give an explicit description of 
the $K$-theory class of $\mathbb{F}_{\rm{norm}}$. 
Let $\Phi$ be the functor defined by 
\begin{align*}
\Phi(-)=\mathrm{Cone}(\pi^{\ast} \dR \pi_{\ast}(-) \to (-)): \ 
D^b(\Coh(X)) \to D^b(\Coh(X)). 
\end{align*}
\begin{lem}\label{lem:spherical}
The functor $\Phi$ is an equivalence such that, 
for any point $y \in Y$, the 
following diagram commutes
\begin{align}\label{diagram:twist}
\xymatrix{
D^b(\Coh(f_y)) \ar[r]^-{i_{\ast}} \ar[d]_-{\Phi_{f_y}} & 
D^b(\Coh(X)) \ar[r]^-{\dL i^{\ast}} \ar[d]_-{\Phi} & D^b(\Coh(f_y)) 
\ar[d]_-{\Phi_{f_y}} \\
D^b(\Coh(f_y)) \ar[r]^-{i_{\ast}} & D^b(\Coh(X)) \ar[r]^-{\dL i^{\ast}} &
D^b(\Coh(f_y)).
}
\end{align}
Here 
$f_y=\pi^{-1}(y)$, 
$i \colon f_y \hookrightarrow X$ is the inclusion 
 and 
$\Phi_{f_y}$ is the spherical twist 
with respect to $\oO_{f_y}$:
\begin{align*}
\Phi_{f_y}(-)=\mathrm{Cone}(\dR \Gamma(-) \otimes \oO_{f_y} \to (-)). 
\end{align*}
\end{lem}
\begin{proof}
Since we have the decomposition
\begin{align}\label{decompose}
\dR \pi_{\ast}\oO_X=\oO_Y \oplus \omega_Y[-1],
\end{align}
the object $\oO_X$ is EZ-spherical in the sense 
of~\cite[Definition~2.1]{Horja}.
Therefore the equivalence of $\Phi$ follows from~\cite[Theorem~2.11]{Horja}. 
The commutative diagram (\ref{diagram:twist}) follows from the definition of $\Phi$. 
\end{proof}
We denote by $\Phi^{\times r}$ the $r$-times 
composition of $\Phi$. 
We define the following object
\begin{align}\label{object:I}
(\Phi^{\times r} \boxtimes \id_X)(\oO_{\Delta_X})[-1] 
\in D^b(\Coh(X \times X)),
\end{align}
and regard it as a family of objects over the second factor of $X \times X$. 
\begin{lem}\label{lem:family}
The object (\ref{object:I}) is a flat family of 
stable sheaves $F$ on $X$ with $[F]=r[f]$ and $\chi(F)=-1$. 
\end{lem}
\begin{proof}
Let us take $x \in X$ and set $y=f(x)$. 
By the diagram (\ref{diagram:twist}),
we have 
\begin{align*}
(\Phi^{\times r} \boxtimes \id_X)(\oO_{\Delta_X})|_{X \times \{x\}}
=i_{\ast}\Phi_{f_y}^{\times r}(\oO_x). 
\end{align*}
It is enough to show that $\Phi_{f_y}^{\times r}(\oO_x)[-1]$
is a stable sheaf on $f_y$ with rank $r$ and Euler characteristic
 $-1$. We prove the claim by an induction on $r$. 
For $r=1$, we have 
$$\Phi_{f_y}(\oO_x)=\oO_{f_y}(-x)[1], $$ 
so the claim is obvious. 
Suppose that the claim holds for $r$. 
By the induction hypothesis,  
$$\dR \Gamma(\Phi_{f_y}^{\times r}(\oO_x)[-1])=\mathbb{C}[-1]. $$ 
So there is a non-split 
extension of sheaves on $f_y$,
\begin{align*}
0 \to 
\Phi_{f_y}^{\times r}(\oO_x)[-1]
\to \Phi_{f_y}^{\times r+1}(\oO_x)[-1] \to \oO_{f_y} \to 0. 
\end{align*}
Since $\Phi_{f_y}^{\times r}(\oO_x)[-1]$ is stable 
and the above sequence is non-split, it is straightforward 
to check that $\Phi_{f_y}^{\times r+1}(\oO_x)[-1]$ is also stable. 
\end{proof}
We take the shifted derived dual of the object (\ref{object:I}):
\begin{align*}
\mathbb{F}^{(r)}:=\dR \mathcal{H}om_{X\times X}((\Phi^{\times r} \boxtimes \id_X)(\oO_{\Delta_X}), \oO_{X \times X})[4] \in D^b(\Coh(X \times X)). 
\end{align*}
Again we regard it as a family of objects over the second factor of 
$X \times X$.
\begin{lem}\label{unnorm univ}
The object $\mathbb{F}^{(r)}$ is a (non-normalized)
universal family of one dimensional stable sheaves $E$ on $X$ with 
$[E]=r[f]$ and $\chi(E)=1$. 
\end{lem}
\begin{proof}
By Lemma~\ref{lem:family}, the object $\mathbb{F}^{(r)}$ is a flat 
family of one dimensional stable sheaves $E$ on $X$ with 
$[E]=r[f]$ and $\chi(E)=1$. Therefore it induces the morphism
$X \to M_{r[f], 1}$. 
Since $\Phi$ is an equivalence, 
the above map is injective on closed points and 
isomorphisms on tangent spaces. 
Since $M_{r[f], 1}\cong X$ (Lemma \ref{lem:fib:isom}), the above 
morphism must be an isomorphism. 
Therefore $\mathbb{F}^{(r)}$ is a universal object. 
\end{proof}
We next determine the $K$-theory class of $\mathbb{F}^{(r)}$. 
We use the following diagram, where the middle square is 
a derived Cartesian
\begin{align*}
\xymatrix{
X \ar@<-0.5ex>@{^{(}->}[r]^-{\iota} 
\ar[d]^-{\pi}\ar@/^20pt/[rr]^-{\Delta_X} &
X \times_Y X \ar[ld]^-{\eta} \ar@<-0.5ex>@{^{(}->}[r]^-{i} \ar[d]_-{p} & X \times X \ar[d]_-{\pi \times \id_X} \ar[rd]^-{\pi\times \pi} & \\
Y\ar@/_20pt/[rrr]_-{\Delta_Y} & X \ar[l]^-{\pi} \ar@<-0.3ex>@{^{(}->}[r]_-{(\pi, \id_X)} & Y \times X 
\ar[r]_-{\id_Y \times \pi\quad} & Y\times Y. 
}
\end{align*}
Here $\iota$ is the diagonal embedding and 
$p$ is the second projection. 
\begin{lem}\label{lem:Ktheory}
We have the following identity in the K-theory of 
$X \times X$
\begin{align*}
[\mathbb{F}^{(r)}]=
-\left(\sum_{j=0}^{r-1}i_{\ast}\eta^{\ast}\omega_Y^{\otimes j}-\oO_{\Delta_X}\right)^{\vee}. 
\end{align*}
\end{lem}
\begin{proof}
By the definition of $\mathbb{F}^{(r)}$, it is enough to show 
the identity in the $K$-theory
\begin{align}\label{id:r}
(\Phi^{\times r} \boxtimes \id_X)(\oO_{\Delta_X})=
-\sum_{j=0}^{r-1}i_{\ast}\eta^{\ast}\omega_Y^{\otimes j}+\oO_{\Delta_X}. 
\end{align}
We prove (\ref{id:r}) by the induction on $r$. 
For $r=1$, since 
\begin{align*}
(\pi \times \id_X)^{\ast}(\pi\times \id_X)_{\ast}\oO_{\Delta_X}
&\cong (\pi\times \id_X)^{\ast}(\pi, \id_X)_{\ast}p_{\ast} \iota_{\ast}\oO_X \\
&\cong i_{\ast}\oO_{X \times_Y X},
\end{align*}
thus we have 
$$(\Phi \boxtimes \id_X)(\oO_{\Delta_X})=i_{\ast}\oO_{X\times_Y X}(-\Delta_X)[1], $$ 
where $\oO_{X\times_Y X}(-\Delta_X) \subset \oO_{X\times_Y X}$ is the ideal 
sheaf of the diagonal. So (\ref{id:r}) holds for $r=1$. 

Suppose that (\ref{id:r}) holds for $r$. 
We have isomorphisms
\begin{align*}
(\pi \times \id_X)^{\ast}(\pi \times \id_X)_{\ast}i_{\ast}\eta^{\ast}\omega_Y^{\otimes j}&\cong 
(\pi \times \id_X)^{\ast}(\pi, \id_X)_{\ast} \dR p_{\ast}\eta^{\ast}\omega_Y^{\otimes j} \\
&\cong 
(\pi \times \id_X)^{\ast}(\pi, \id_X)_{\ast}(\pi^{\ast}\omega_Y^{\otimes j}
\oplus \pi^{\ast}\omega_Y^{\otimes j+1}[-1]) \\
&\cong i_{\ast}p^{\ast}(\pi^{\ast}\omega_Y^{\otimes j}
\oplus \pi^{\ast}\omega_Y^{\otimes j+1}[-1]).
\end{align*}
Here we have used (\ref{decompose}) for the second isomorphism. 
It follows that 
the object $(\Phi \boxtimes \id_X)(i_{\ast}\eta^{\ast}\omega_Y^{\otimes j})$
has $K$-theory class $i_{\ast}\eta^{\ast}\omega_Y^{\otimes j+1}$. 
By the induction hypothesis, we have 
\begin{align*}
(\Phi^{\times r+1} \boxtimes \id_X)(\oO_{\Delta_X})=
-\sum_{j=0}^{r-1}i_{\ast}\eta^{\ast}\omega_Y^{\otimes j+1}
-i_{\ast}\oO_{X\times_Y X} +\oO_{\Delta_X},
\end{align*}
which implies (\ref{id:r}) holds for $r+1$. 
\end{proof}

We define $\mathbb{F}^{(r)}_{\rm{norm}}$ to be
\begin{align*}
\mathbb{F}^{(r)}_{\rm{norm}}
\cneq \mathbb{F}^{(r)} \otimes p_M^{\ast}\pi^{\ast}\omega_Y^{\otimes r}
\in \Coh(X \times X). 
\end{align*}
Here $p_M \colon X \times X \to X$ is the second projection. 
\begin{lem}\label{lem:Fnorm}
The object $\mathbb{F}^{(r)}_{\rm{norm}}$
is the normalized universal sheaf of 
one dimensional stable sheaves $E$ on $X$ with $[E]=r[f]$ and $\chi(E)=1$. 
\end{lem}
\begin{proof}
By a similar computation in Lemma~\ref{lem:Ktheory}
and the Grothendieck duality, 
we have 
\begin{align*}
\dR p_{M\ast}(i_{\ast}\eta^{\ast}\omega_Y^{\otimes j})^{\vee}
\cong \pi^{\ast}\omega_Y^{\otimes -j}\oplus 
\pi^{\ast}\omega_Y^{\otimes -j-1}[1]. 
\end{align*}
Therefore 
by Lemma~\ref{lem:Ktheory} we have an identity in the $K$-theory: 
\begin{align*}
\dR p_{M\ast} \mathbb{F}^{(r)}
&=-\sum_{j=0}^{r-1}(\pi^{\ast}\omega_Y^{\otimes -j}-\pi^{\ast}\omega_Y^{\otimes
-j-1})+\oO_X \\
&\cong \pi^{\ast}\omega_Y^{\otimes -r}. 
\end{align*}
Therefore the lemma holds from the definition of normalized 
universal sheaf. 
\end{proof}

Using above lemmas, 
we can compute $\langle \tau_1(\alpha) \rangle_{r[f]}$
in the following. 
\begin{prop}\label{prop:elliptic}
Let $B,D\in H^2(X,\mathbb{Z})$ be divisors in \eqref{div:BE}.
For $\alpha=\alpha_1[D]+\alpha_2[B]$, 
we have 
\begin{align*}
\langle \tau_1(\alpha) \rangle_{r[f]}
=(20-1920r^2)\alpha_1+960 \alpha_2, \quad\forall\,\, \alpha_1,\alpha_2\in \mathbb{R}.
\end{align*}
Here we have chosen the minus sign for 
the virtual class (\ref{Mvir}). 
\end{prop}
\begin{proof}
Using Lemma~\ref{lem:Ktheory} and 
Grothendieck Riemann-Roch theorem, we can easily
calculate the Chern character of $\mathbb{F}^{(r)}_{\rm{norm}}$
as follows
(we only describe the terms with $\ch_3$ and $\ch_4$)
\begin{align*}
\ch(\mathbb{F}_{\rm{norm}}^{(r)})&=
-\left(\sum_{j=0}^{r-1}\ch(i_{\ast}\eta^{\ast}\omega_Y^{\otimes j})
-\ch(\oO_{\Delta_X})\right)^{\vee} \cdot 
e^{-4rp_M^{\ast}\pi^{\ast}h} \\
&=-\left(\sum_{j=0}^{r-1}i_{\ast}(e^{-4j\eta^{\ast}h}\cdot 
\eta^{\ast}
 \mathrm{td}_{Y/Y\times Y}^{-1})-\ch(\oO_{\Delta_X})\right)^{\vee}
\cdot 
e^{-4rp_M^{\ast}\pi^{\ast}h} \\
&=-\left\{  \sum_{j=0}^{r-1}i_{\ast}(1-4j\eta^{\ast}h-2\eta^{\ast}h+\cdots)^{\vee}+\Delta_X+ \cdots\right\}\cdot 
e^{-4rp_M^{\ast}\pi^{\ast}h} \\
&=\Big(r[X\times_Y X]+2r^2 i_{\ast}\eta^{\ast}h+\Delta_X+\cdots\Big) \cdot 
e^{-4rp_M^{\ast}\pi^{\ast}h}\\
&=r[X\times_Y X]-2r^2 i_{\ast}\eta^{\ast}h+\Delta_X+\cdots.
\end{align*}
Therefore we have 
\begin{align*}
p_{M\ast}(\ch_4(\mathbb{F}^{(r)}_{\rm{norm}}) \cdot p_X^{\ast}\alpha)
=\alpha_1(D-2r^2B)+\alpha_2 B.
\end{align*}
The proposition follows from $D \cdot c_3(X)=-20$ and $B \cdot c_3(X)=-960$. 
Here we note that the divisor $E$ in \cite[pp. 38-39]{KP} satisfies $E=D+4B$, then 
$D \cdot c_3(X)=-20$ follows
from $E \cdot c_3(X)=-3860$ and $B \cdot c_3(X)=-960$. 
\end{proof}
We next compute the meeting numbers. 
\begin{lem}
For any $r_1>0$, $r_2>0$, we have 
\begin{align}\label{id:meet} 
m_{r_1[f], r_2[f]}=46080.
\end{align}
\end{lem}
\begin{proof}
We set $\gamma_1=D \cdot B$ and $\gamma_2=B^2$. 
By~\cite[Section~6.3]{KP}, we have 
$$n_{0, r[f]}(\gamma_1)=960r, \quad n_{0, r[f]}(\gamma_2)=0. $$ 
As in \cite[Section~6.3]{KP}, to compute $n_{0,\beta}(\gamma)$ for some $\gamma\in H^4(X)$, it is enough to 
take $\gamma_1$, $\gamma_2$ as a basis and consider projections of $\gamma$ to them.
Therefore, we write 
\begin{align}\label{diagonal}
\Delta=\gamma_1\boxtimes \gamma_2+\gamma_2 \boxtimes \gamma_1-4\gamma_2 \boxtimes \gamma_2, \quad 
c_2(X)=48\gamma_1+182\gamma_2. 
\end{align}
We prove (\ref{id:meet}) by an induction on $r_1+r_2$. 
If $r_1=r_2=1$, we have 
\begin{align*}
m_{[f], [f]}=n_{0, [f]}(c_2(X))=48 \cdot 960=46080.
\end{align*}
For general $(r_1, r_2)$, suppose that $r_1>r_2$. 
Then because of $n_{0, r_i[f]}(\gamma_2)=0$
and the description of the diagonal (\ref{diagonal}), we have 
\begin{align*}
m_{r_1[f], r_2[f]}=m_{(r_1-r_2)[f], r_2[f]}=46080,
\end{align*}
by the induction hypothesis. 
The same applies to the case of $r_1<r_2$. 
Finally for $r_1=r_2=r$, 
using the induction hypothesis we obtain 
\begin{align*}
m_{r[f], r[f]}&=n_{0, r[f]}(c_2(X))-\sum_{\begin{subarray}{c}r'+r''=r \\  r'>0, r''>0 \end{subarray}}m_{r'[f], r''[f]} \\   
&=46080r-46080(r-1) \\
&=46080.
\end{align*}
Therefore the identity (\ref{id:meet}) holds. 
\end{proof}

\begin{thm}\label{thm:elliptic}
Conjecture~\ref{main conj} holds for $\beta=r[f]$ with $r\geqslant 1$. 
\end{thm}
\begin{proof}
For $\alpha=\alpha_1[D]+\alpha_2[B]$
with $\alpha_1, \alpha_2 \in \mathbb{R}$, 
we have 
\begin{align*}
&\frac{n_{0, r[f]}(\alpha^2)}{2(\alpha \cdot r[f])}
-\frac{1}{4}\sum_{\begin{subarray}{c}r_1+r_2=r \\  r_i\in \mathbb{Z}_{>0} \end{subarray}}
\frac{(\alpha \cdot r_1[f])(\alpha \cdot r_2[f])}{(\alpha \cdot r[f])}m_{r_1[f], r_2[f]}-\sum_{k|r}(\alpha \cdot r[f]/k)n_{1, r[f]/k} \\
&=960(-2\alpha_1+\alpha_2)-\frac{1}{4}\sum_{\begin{subarray}{c}r_1+r_2=r \\  r_i\in \mathbb{Z}_{>0} \end{subarray}}
\frac{r_1 \alpha_1 \cdot r_2 \alpha_1}{r\alpha_1}\cdot 46080
+20\alpha_1 \\
&=960(-2\alpha_1+\alpha_2)-\frac{1}{4} \cdot 
\frac{1}{6}(r^2-1)\alpha_1\cdot 46080 
+20\alpha_1 \\
&=(20-1920r^2)\alpha_1+960\alpha_2. 
\end{align*}
Here we have used $n_{1, [f]}=-20$ and $n_{1, r[f]}=0$ 
for $r>1$ from~\cite[Table~7]{KP}. 
By comparing with Proposition~\ref{prop:elliptic}, we obtain the theorem. 
\end{proof}

\subsection{Product of CY 3-fold with elliptic curve}
We next consider the case of a CY 4-fold given by the 
product of a smooth projective CY 3-fold and an elliptic curve. 
\begin{prop}\label{prop on product}
Let $X=Y\times E$ be the product of a smooth projective Calabi-Yau 3-fold $Y$ with an elliptic curve $E$. Then 
Conjecture \ref{main conj} holds in the following cases:
\begin{enumerate}
\item 
$\beta=r[E]$ with $r\geqslant 1$, 
\item any $\beta\in H_2(Y,\mathbb{Z})\subset H_2(X,\mathbb{Z})$.
\end{enumerate}
\end{prop}
\begin{proof}
(1) By \cite[Lemma 2.16]{CMT2}, genus zero GW invariants vanish and 
$$n_{1,[E]}=\chi(Y), \quad n_{1,r[E]}=0, \quad \mathrm{if}\,\, r>1. $$
We have isomorphisms 
$$M_1(X,[E])\cong Y\times \Pic^1(E), \quad M_1(X,r[E])=\emptyset, \quad \mathrm{if}\,\, r>1, $$
whose virtual class satisfies 
$$[M_1(X,[E])]^{\mathrm{vir}}=-\chi(Y)\cdot[\Pic^1(E)], \quad [M_1(X,r[E])]^{\mathrm{vir}}=0, \quad \mathrm{if}\,\, r>1, $$
for certain choice of orientations.
Then Conjecture \ref{main conj} obviously holds when $r>1$. 

As for $r=1$ case, let $\Delta_Y\hookrightarrow Y\times Y$ and $\Delta_E\hookrightarrow E\times E$ be the diagonals. Then  
$$\mathbb{F}=\oO_{\Delta_Y}\boxtimes \oO_{E\times E}(\Delta_E) $$
is the normalized universal sheaf of $M_1(X,[E])$. 
By the Grothendieck-Riemann-Roch formula,  
$$\ch_4(\mathbb{F})=[\Delta_Y\times \Delta_E]. $$
For $\alpha=\alpha_1\oplus \alpha_2\in H^2(Y)\oplus H^2(E)$, we have 
$\tau_1(\alpha)=\alpha_1\oplus \alpha_2$. Therefore
$$\int_{[M_1(X,[E])]^{\mathrm{vir}}}\tau_1(\alpha)=-\chi(Y)\cdot \int_E\alpha_2=-(\alpha\cdot [E])\,n_{1,[E]}.$$
(2) By \cite[Lemma 2.16]{CMT2}, $n_{1,\beta}=0$. By \cite[Lemma 2.6]{CMT1}, we have an isomorphism
$$M_1(X,\beta)\cong M_1(Y,\beta)\times E, $$
under which the virtual class satisfies 
$$[M_1(X,\beta)]^{\mathrm{vir}}=[M_1(Y,\beta)]^{\mathrm{vir}}\otimes [E], $$
for certain choice of orientations.

Let $\mathbb{F}_Y$ be the normalized universal sheaf of $M_1(Y,\beta)$ and $\Delta_E\hookrightarrow E\times E$ be the diagonal embedding. Then 
$$\mathbb{F}=\mathbb{F}_Y\boxtimes \oO_{\Delta_E}$$
is the normalized universal sheaf of $M_1(X,\beta)$ whose Chern character satisfies 
$$\ch_3(\mathbb{F})=\ch_2(\mathbb{F}_Y)\cdot [\Delta_E], \quad \ch_4(\mathbb{F})=\ch_3(\mathbb{F}_Y)\cdot [\Delta_E]. $$
Let $\pi_{M_Y}:M_1(Y,\beta)\times Y\to M_1(Y,\beta)$ be the projection. Then 
\begin{align*}
\int_{[M_1(X,\beta)]^{\mathrm{vir}}}\tau_1(\alpha)
&=\int_E\alpha_2\cdot \int_{[M_1(Y,\beta)]^{\mathrm{vir}}}(\pi_{M_Y})_*\ch_3(\mathbb{F}_Y) \\
&=\int_E\alpha_2\cdot \int_{[M_1(Y,\beta)]^{\mathrm{vir}}}\ch_0((\pi_{M_Y})_*\mathbb{F}_Y) \\
&=\bigg(\int_E\alpha_2\bigg)\cdot\deg[M_1(Y,\beta)]^{\mathrm{vir}},  
\end{align*}
where the second equality is by the Grothendieck-Riemann-Roch formula.

By \cite[Corollary 2.7]{CMT1}, for certain choice of orientations, we have 
\begin{align*}
\int_{[M_1(X,\beta)]^{\mathrm{vir}}}\tau_0(\alpha^2)=2\bigg(\int_\beta\alpha\bigg)\cdot \bigg(\int_E\alpha_2\bigg)\cdot \deg[M_1(Y,\beta)]^{\mathrm{vir}}. \end{align*}
Note that meeting invariants vanish obviously in this product case, therefore Conjecture \ref{main conj} holds in this case.
\end{proof}

\subsection{Two local surfaces}
In this section, we consider local surfaces 
\begin{align*}X=\mathrm{Tot}_{\mathbb{P}^2}(\oO(-1)\oplus \oO(-2)), \quad 
\mathrm{Tot}_{\mathbb{P}^1\times \mathbb{P}^1}(\oO(-1,-1)\oplus \oO(-1,-1)). \end{align*}
Although they are non-compact, the moduli spaces $M_1(X,\beta)$ of one dimensional stable sheaves are smooth projective varieties
(ref. \cite[Section 3.1]{CMT1}). So it makes sense to study Conjecture \ref{main conj} for these examples.

\begin{prop}\label{prop on local P2}
Let $X=\mathrm{Tot}_{\mathbb{P}^2}(\oO(-1)\oplus \oO(-2))$ and $\beta=d\,[\mathbb{P}^1]\in H_2(X,\mathbb{Z})$.
Then Conjecture \ref{main conj} holds for $d\leqslant 3$.
\end{prop}
\begin{proof}
We present the proof for $d=3$ case (the $d=1,2$ cases are easier versions of the same approach).
We first compute the meeting invariants, consider a compactification of $X$: 
$$\overline{X}=\mathbb{P}(\oO(-2)\oplus \oO(-1)\oplus \oO)$$
with the projection $\pi:\overline{X}\to \mathbb{P}^2$ and a section $\iota: \mathbb{P}^2\to \overline{X}$. Then 
$$H^4(\overline{X},\mathbb{Z})=\mathbb{Z}\langle T_1,T_2,T_3 \rangle, $$ 
where $T_1=[\iota(\mathbb{P}^2)]$, $T_2=[\pi^{-1}(\mathrm{pt})]$ and $T_3$ is the class of a surface in the infinite divisor. Their intersection matrix is 
\begin{align*}
g=\begin{pmatrix}
2 & 1 & 0 \\
1 & 0 & 0 \\
0 & 0 & \ast
\end{pmatrix}
\quad \mathrm{with}  \quad 
g^{-1}=\begin{pmatrix}
0 & 1 & 0 \\
1 & -2 & 0 \\
0 & 0 & (\ast)^{-1}
\end{pmatrix}.
\end{align*}
For any $\beta\in H_2(X)\subset H_2(\overline{X})$, we have 
$$n_{0,\beta}(T_1)=(T_1\cdot T_1)\cdot n_{0,\beta}([\mathrm{pt}]), \quad 
n_{0,\beta}(T_2)=n_{0,\beta}([\mathrm{pt}]),\quad n_{0,\beta}(T_3)=0. $$
Then
$$m_{1,1}=n_{0,1}(c_2(X))+\sum n_{0,1}(T_i)g^{ij}n_{0,1}(T_j)=6, $$
$$m_{1,2}=m_{1,1}+\sum n_{0,1}(T_i)g^{ij}n_{0,2}(T_j)=4. $$
So for $\alpha=[\mathbb{P}^1]\in H^2(X)$, the RHS of Conjecture \ref{main conj} is  
$$\frac{n_{0,3}([\mathrm{pt}])}{6}-\frac{1}{3}m_{1,2}-3\,n_{1,3}=\frac{3}{2},$$
where we use $n_{0,3}([\mathrm{pt}])=n_{1,3}=-1$ (ref. \cite[pp. 22]{KP}). 

Next we compute the $\DT_4$ descendent invariant. By \cite[Prop. 3.1]{CMT1}, there is an isomorphism 
\begin{align}\label{isom:P2}
M_1(\mathbb{P}^2,d) \stackrel{\cong}{\to}
M_1(X,d)
\end{align}
given by push-forward along the zero section 
$\mathbb{P}^2 \hookrightarrow X$. 
Under the above isomorphism, 
the virtual class can be computed as the Euler class of a 
certain obstruction bundle on $M_1(\mathbb{P}^2,d)$.

More specifically, there is a natural support morphism to the linear system of degree $3$ curves
$$M_1(\mathbb{P}^2,3) \rightarrow |\mathcal{O}(3)| = \mathbb{P}^9.$$  
Moreover, if we denote 
$$\mathcal{C} \hookrightarrow \mathbb{P}^9 \times \mathbb{P}^2$$
to be the universal curve over this linear system,
there is an isomorphism
$\mathcal{C} \cong M_1(\mathbb{P}^2,3)$
which sends the pair $(C,p)$ to the dual (on $C$) of the ideal sheaf $I_{C,p}$.

Let $\mathcal{V}$ denote the vector bundle on $M_1(\mathbb{P}^2,3)$ whose fiber at a point $[E]$ is 
$$\Ext^1_{\mathbb{P}^2}(E,E(-1)) = \RHom_{\mathbb{P}^2}(I_{C,p},I_{C,p}(-1))[1].$$  
Its top Chern class can be computed via the diagram:
\begin{align*}
\xymatrix{
\cC \times \mathbb{P}^2
\ar@<-0.3ex>@{^{(}->}[r]^-{j}
\ar[d]_{\pi_{\cC}}\ar@{}[dr]|\square &
 \mathbb{P}^9 \times \mathbb{P}^2 \times \mathbb{P}^2
 \ar[d]^{\pi_{1, 2}} \\
\cC \ar@<-0.3ex>@{^{(}->}[r]^-{j} & \mathbb{P}^9 \times \mathbb{P}^2. 
}
\end{align*}
By Beilinson's resolution of diagonal, the $K$-theory class of the universal ideal sheaf
$\iI$ on $\mathcal{C}\times\mathbb{P}^2$ is given by the pullback
$[\mathcal{I}] = j^{\ast}[\fF]$, 
where $[\fF]$ is 
the $K$-theory class of $\mathbb{P}^9 \times \mathbb{P}^2 \times \mathbb{P}^2$
given by
\begin{align*}
[\fF] &= \pi_{1,3}^{\ast}[\mathcal{O}_{\mathcal{C}}] - \pi_{2,3}^{*}[\mathcal{O}_{\Delta}] \\
&= 3[\mathcal{O}(0,-1,0)] - [\mathcal{O}(-1,0,-3)] - [\mathcal{O}(0,-1,1)] - [\mathcal{O}(0,-2,-1)].
\end{align*}
The $K$-theory class of the corresponding universal sheaf of $M_1(\mathbb{P}^2,3)$ is the pull-back of $[\fF^{\vee}\otimes\oO(0,0,-3)[1]]$.
Its push-forward to the moduli space
$M_1(\mathbb{P}^2,3)$
 has determinant $\oO(1,0)$. Therefore we define normalized universal sheaf (as an element in $K$-theory):
\begin{align*}[\eE]&=-[\fF^{\vee}\otimes\oO(0,0,-3)\otimes\oO(-1,0,0)] \\
&=-3[\oO(-1,1,-3)]+[\oO(0,0,0)]+[\oO(-1,1,-4)]+[\oO(-1,2,-2)].
\end{align*}
From this, we obtain 
\begin{align*}
[\mathcal{V}] =j^{\ast}\pi_{1,2, \ast}(-[\eE]^{\vee}\otimes[\eE]\otimes\mathcal{O}(0,0,-1))=(H_1^7\cdot H_2^2)|_{\mathcal{C}}.
\end{align*}
Moreover $\ch_4(\eE)$
is written as
\begin{align*}
\ch_4(\eE) &=-\frac{3}{2}H_1\cdot H_3+\mathrm{other}\,\,\mathrm{terms},
\end{align*}
where the other terms do not contribute to the descendent 
$\DT_4$-invariant. 
Therefore 
 \begin{align*}
 \int_{[M_1(X,3)]^{\mathrm{vir}}}\tau_1([\mathbb{P}^1])
 =(-1)\cdot \int_{\mathbb{P}^9\times \mathbb{P}^2}(H_1^7\cdot H_2^2)(-\frac{3}{2}H_1\cdot H_3)(H_1+3H_2) =\frac{3}{2}, 
\end{align*}
where the minus sign $(-1)^{c_1(Y)\cdot\beta-1}=-1$ (here $Y=\mathrm{Tot}_{\mathbb{P}^2}(\oO(-1))$) is the preferred choice of orientations \eqref{prefer choice}.
Therefore Conjecture \ref{main conj} holds in this case.
\end{proof}

\begin{prop}\label{prop on local P1 product} 
Let $X=\mathrm{Tot}_{\mathbb{P}^1\times \mathbb{P}^1}(\oO(-1,-1)\oplus \oO(-1,-1))$ and $\beta=(d_1,d_2)\in H_2(X,\mathbb{Z})$.
Then Conjecture \ref{main conj} holds for $d_1,d_2\leqslant 2$.
\end{prop}
\begin{proof}
Similar to the proof of Proposition \ref{prop on local P2}, we present the proof for $(2,2)$ class here. 
Consider a compactification of $X$: 
$$\overline{X}=\mathbb{P}(\oO(-1,-1)\oplus \oO(-1,-1)\oplus \oO), $$
with the projection $\pi:\overline{X}\to \mathbb{P}^1\times \mathbb{P}^1$ and a section $\iota: \mathbb{P}^1\times \mathbb{P}^1\to \overline{X}$. Then 
$$H^4(\overline{X},\mathbb{Z})=\mathbb{Z}\langle T_1,T_2,T_3 \rangle, $$ 
where $T_1=[\iota(\mathbb{P}^1\times \mathbb{P}^1)]$, $T_2=[\pi^{-1}(\mathrm{pt})]$ and $T_3$ is the class of a surface in the infinite divisor. Their intersection matrix is 
\begin{align*}
g=\begin{pmatrix}
2 & 1 & 0 \\
1 & 0 & 0 \\
0 & 0 & \ast
\end{pmatrix}
\quad \mathrm{with}  \quad 
g^{-1}=\begin{pmatrix}
0 & 1 & 0 \\
1 & -2 & 0 \\
0 & 0 & (\ast)^{-1}
\end{pmatrix}.
\end{align*}
For any $\beta\in H_2(X)\subset H_2(\overline{X})$, we have 
$$n_{0,\beta}(T_1)=(T_1\cdot T_1)\cdot n_{0,\beta}([\mathrm{pt}]), \quad 
n_{0,\beta}(T_2)=n_{0,\beta}\left([\mathrm{pt}]\right),\quad n_{0,\beta}(T_3)=0. $$
By the definition of meeting invariants, we have 
\begin{align*}
&m_{(0,1),(2,1)}=2\,n_{0,(0,1)}\left([\mathrm{pt}]\right)n_{0,(2,1)}\left([\mathrm{pt}]\right)+2\,n_{0,(0,1)}\left([\mathrm{pt}]\right)n_{0,(2,0)}\left([\mathrm{pt}]\right)=2, \\ 
& m_{(0,2),(2,0)}=2\,n_{0,(0,2)}([\mathrm{pt}])^2=0,  \\
&m_{(1,0),(1,2)}=2\,n_{0,(1,0)}([\mathrm{pt}])n_{0,(1,2)}([\mathrm{pt}])+2\,n_{0,(1,0)}([\mathrm{pt}])n_{0,(0,2)}([\mathrm{pt}])=2, \\
&m_{(1,1),(1,1)}=-2\,n_{0,(1,1)}([\mathrm{pt}])+2\,n_{0,(1,1)}([\mathrm{pt}])^2-4\,n_{0,(1,0)}([\mathrm{pt}])^2=-4.
\end{align*}
For $\alpha=(a,b)\in H^2(X)$, the RHS of Conjecture \ref{main conj} is 
\begin{align*}
\frac{2ab\cdot n_{0,(2,2)}([\mathrm{pt}])}{4(a+b)}-\frac{1}{8(a+b)}\sum_{\beta_1+\beta_2=(2,2)}(\alpha\cdot\beta_1)(\alpha\cdot\beta_2)m_{\beta_1,\beta_2}- 2(a+b)n_{1,(2,2)}=-2(a+b),
\end{align*}
where we use $n_{0,(2,2)}([\mathrm{pt}])=2$, $n_{1,(2,2)}=1$ (ref. \cite[pp. 24]{KP}).

Next, we compute the $\DT_4$ descendent invariant. 
The method is similar to the proof of Proposition \ref{prop on local P2}. 
Similarly to (\ref{isom:P2}), we have an isomorphism 
\begin{align*}
M_1(\mathbb{P}^1\times\mathbb{P}^1 ,\beta) \stackrel{\cong}{\to}
M_1(X,\beta).
\end{align*}
For $\beta=(2,2)$, it is isomorphic to the $(1,2,2)$ divisor 
\begin{align*}
\mathcal{C}\hookrightarrow \mathbb{P}^8\times \mathbb{P}^1\times \mathbb{P}^1. 
\end{align*}
We have a commutative diagram 
\begin{align*}
\xymatrix{
\cC \times(\mathbb{P}^1\times\mathbb{P}^1)  
\ar@<-0.3ex>@{^{(}->}[r]^-{j}
\ar[d]_{\pi_{\cC}}\ar@{}[dr]|\square  &
 \mathbb{P}^8 \times (\mathbb{P}^1\times\mathbb{P}^1) \times (\mathbb{P}^1\times\mathbb{P}^1)
 \ar[d]^{\pi_{1, 2}} \\
\cC \ar@<-0.3ex>@{^{(}->}[r]^-{j} & \mathbb{P}^8 \times (\mathbb{P}^1\times\mathbb{P}^1). 
}
\end{align*}
The $K$-theory class of the normalized universal sheaf of $M_1(X,\beta)$ 
is the pullback of the following $K$-theory class of 
$\mathbb{P}^8 \times (\mathbb{P}^1 \times \mathbb{P}^1) \times 
(\mathbb{P}^1 \times \mathbb{P}^1)$, 
\begin{align*}
[\eE]=[\oO]-[\oO(-1,1,0,-1,-2)]-[\oO(-1,0,1,-2,-1)]+[\oO(-1,1,1,-1,-1)]. 
\end{align*}
From this, we can compute the virtual class 
$$[M_1(X,(2,2))]^{\mathrm{vir}}=-2H_1^6H_2H_3|_{\mathcal{C}}, $$
for certain choice of orientations and 
\begin{align*}
\int_{[M_1(X,(2,2))]^{\mathrm{vir}}}\tau_1(\alpha)&=
(-1)\int_{\mathbb{P}^8 \times\mathbb{P}^1 \times\mathbb{P}^1}(-2)H_1^6H_2H_3(H_1+2H_2+2H_3)(a+b)(-H_1) \\
&=-2(a+b),
\end{align*}
where the minus sign $(-1)^{c_1(Y)\cdot\beta-1}=-1$ (here $Y=\mathrm{Tot}_{\mathbb{P}^1\times \mathbb{P}^1}(\oO(-1,-1))$) is the preferred choice of orientations \eqref{prefer choice}.
Therefore Conjecture \ref{main conj} holds in this case.
\end{proof}

\subsection{Local Fano 3-fold}

\begin{prop}\label{local Fano 3}
Let $X=K_{\mathbb{P}^3}$ and $\beta=d\,[\mathbb{P}^1]\in H_2(X,\mathbb{Z})$. Then Conjecture \ref{main conj} holds if $d\leqslant 3$.
\end{prop}
\begin{proof}
We first consider $d=2$ case (same method applies to $d=1$ case). 
Consider a compactification of $X$: 
$$\overline{X}=\mathbb{P}(\oO_{\mathbb{P}^3}(-4) \oplus \oO_{\mathbb{P}^3})$$
with the projection $\pi:\overline{X}\to \mathbb{P}^3$ and a section $\iota: \mathbb{P}^3\to \overline{X}$. Then 
$$H^4(\overline{X},\mathbb{Z})=\mathbb{Z}\langle T_1,T_2 \rangle, $$ 
where $T_1=[\iota(\mathbb{P}^2)]$, $T_2=[\pi^{-1}(\mathbb{P}^1)]$. Their intersection matrix is 
\begin{align*}
g=\begin{pmatrix}
-4 & 1   \\
1 & 0 
\end{pmatrix}
\quad \mathrm{with}  \quad 
g^{-1}=\begin{pmatrix}
0 & 1   \\
1 & 4 
\end{pmatrix}.
\end{align*}
For any $\beta\in H_2(X)\subset H_2(\overline{X})$, we have 
$$n_{0,\beta}(T_1)=(T_1\cdot T_1) \,n_{0,\beta}\left([\mathbb{P}^1]\right), \quad 
n_{0,\beta}(T_2)=n_{0,\beta}\left([\mathbb{P}^1]\right). $$
By a direct calculation, we have 
$$c_2(X)=-10\in  \mathbb{Z}\cong H^4(X,\mathbb{Z}).$$
Then meeting invariants can be calculated as 
\begin{align*}
&m_{1,1}=n_{0,1}\left(c_2(X)\right)+\sum n_{0,1}(T_i)g^{ij}n_{0,1}(T_j)=-1400, 
\\
&m_{1,2}=m_{1,1}+\sum n_{0,1}(T_i)g^{ij}n_{0,2}(T_j)=-67000, 
\end{align*}
where we use (ref. \cite[Table 1, pp. 31]{KP}):
$$n_{0,1}\left([\mathbb{P}^1]\right)=-20, \quad n_{0,2}\left([\mathbb{P}^1]\right)=-820, \quad  
n_{0,3}\left([\mathbb{P}^1]\right)=-68060. $$
Hence for $\alpha=[\mathbb{P}^2]\in H^2(X)$, the RHS of Conjecture \ref{main conj} in this case is
\begin{align}
\label{-30}
\frac{1}{4}n_{0,2}([\mathbb{P}^1])-\frac{1}{8}m_{1,1}=-30. 
\end{align}
Next we compute the $\DT_4$ descendent invariant. 
It is easy to see that for
any one dimensional stable sheaf $F$ on $\mathbb{P}^3$
with $[F]=2[\mathbb{P}^1]$ and $\chi(F)=1$, 
there is an unique plane $H \subset \mathbb{P}^3$ and a conic 
$C \subset H$ such that $F \cong \oO_C$. 
Therefore $M_{1}(\mathbb{P}^3, 2)$ is a projective bundle
over $|\oO_{\mathbb{P}^3}(1)| \cong \mathbb{P}^3$ and we denote the projection by 
\begin{align*}
g \colon 
M_1(\mathbb{P}^3, 2) \cong \mathbb{P}_{|\oO_{\mathbb{P}^3}(1)|}
(\eE) \to 
|\oO_{\mathbb{P}^3}(1)|. 
\end{align*}
Here $\eE=f_{\ast}\oO_{\hH}(2)$ and 
$\hH$ is the universal hyperplanes in the following diagram 
\begin{align*}
\xymatrix{
\hH \ar@<-0.3ex>@{^{(}->}[r] \ar[d]_-{f} 
& \mathbb{P}^3 \times |\oO_{\mathbb{P}^3}(1)| \ar[ld] \\
|\oO_{\mathbb{P}^3}(1)|. & 
}
\end{align*}
Let $\cC$ be the relative universal conics
\begin{align*}
i \colon 
\cC \hookrightarrow \hH \times_{|\oO_{\mathbb{P}^3}(1)|} \mathbb{P}_{|\oO_{\mathbb{P}^3}(1)|}(\eE) \hookrightarrow \mathbb{P}^3 \times 
\mathbb{P}_{|\oO_{\mathbb{P}^3}(1)|}(\eE). 
\end{align*}
Then the normalized universal sheaf 
$\mathbb{F}_{\rm{norm}}$ is given by $i_{\ast}\oO_{\cC}$, which 
has the following $K$-theory class
\begin{align}\label{univ:K}
i_{\ast}\oO_{\cC}=[\oO]-[\oO(-1, -h_2)]-[\oO(-2, -h_1)]+[\oO(-3, -h_1-h_2)],
\end{align}
where $h_1=c_1(\oO_{g}(1))$ and $h_2=g^{\ast}\oO(1)$. 
A direct calculation shows 
the identity in $K$-theory
\begin{align*}
&\quad \,\, \dR \pi_{M\ast} \dR \hH om_{\mathbb{P}^3 \times \mathbb{P}_{|\oO_{\mathbb{P}^3}(1)|}(\eE)}(i_{\ast}\oO_{\cC}, i_{\ast}\oO_{\cC})[1]
\\
&=-4[\oO]+8[\oO(h_2)]+20[\oO(h_1)]-4[\oO(h_1-h_2)]-20[\oO(h_1+h_2)]. 
\end{align*}
Since $\mathbb{P}_{|\oO_{\mathbb{P}^3}(1)|}
(\eE)$ is smooth whose tangent bundle has the 
$K$-theory class:
\begin{align*}
T_{\mathbb{P}_{|\oO_{\mathbb{P}^3}(1)|}
(\eE)}=4[\oO(h_2)]+10[\oO(h_1)]-4[\oO(h_1-h_2)]-2[\oO],
\end{align*}
the obstruction bundle $\mathrm{Obs}$
for $\mathrm{DT}_3$ virtual class on $M_1(\mathbb{P}^3, 2)$
has the following $K$-theory class
\begin{align*}
\mathrm{Obs} &=T_{\mathbb{P}_{|\oO_{\mathbb{P}^3}(1)|}}
+\dR \pi_{M\ast} \dR \hH om_{\mathbb{P}^3 \times \mathbb{P}_{|\oO_{\mathbb{P}^3}(1)|}(\eE)}(i_{\ast}\oO_{\cC}, i_{\ast}\oO_{\cC})
-[\oO] \\
&=20 [\oO(h_1+h_2)]+[\oO] -10 [\oO(h_1)]-4 [\oO(h_2)]. 
\end{align*}
By taking the Euler class, the 
$\mathrm{DT}_3$ virtual class on $M_1(\mathbb{P}^3, 2)$
is given by 
\begin{align*}
[M_1(\mathbb{P}^3, 2)]^{\rm{vir}}=120 h_1^7+840 h_1^6 h_2+3080h_1^5 h_2^2+7700 h_1^4 h_2^3. 
\end{align*}
On the other hand for $\alpha=[\mathbb{P}^2]\in H^2(X)$, 
using (\ref{univ:K}), we easily compute
\begin{align*}
\pi_{M\ast}(\alpha \boxtimes \ch_4(i_{\ast}\oO_C))
=-\frac{1}{2}h_1. 
\end{align*}
Therefore
\begin{align*}
\int_{[M_1(\mathbb{P}^3, 2)]^{\rm{vir}}}
\tau_1(\alpha)
&=\left(120 h_1^7+840 h_1^6 h_2+3080h_1^5 h_2^2+7700 h_1^4 h_2^3\right) \cdot \left(-\frac{1}{2}h_1\right) \\
&=-60h_1^8-420h_1^7 h_2-1540h_1^6h_2^2-3850 h_1^5 h_2^3 \\
&=30. 
\end{align*}
Here we have used equalities
$h_1^8=-4$, $h_1^7 h_2=6$, $h_1^6 h_2^2=-4$ and $h_1^5 h_2^3=1$. 
By taking the sign $(-1)^{c_{1}(\mathbb{P}^3)\cdot \beta-1}=-1$ as the preferred choice of orientations \eqref{prefer choice}, 
we obtain a matching with (\ref{-30}). 

Finally when $d=3$, we consider the PT moduli space $P_1(X,3)$ of stable pairs 
$(F,s)$ on $X$, where $F$ is a pure 1-dimensional sheaf with $[F]=3[\mathbb{P}^1]$ and $\chi(F) =1$,  
$s \in H^0(F)$ is a section with at most 0-dimensional cokernel. 
If $F$ is supported on a twisted cubic $C$, then $F=\oO_C$ is a stable sheaf \cite[Lemma 3.1]{FT}. 
If $F$ is supported on a plane cubic $C$, then $F$ is a non-split extension of $\oO_C$ with the structure sheaf of one point, hence 
also stable. 

Therefore there is a well-defined forgetful map 
\begin{align*}
f \colon P_1(X, \beta)\to M_1(X, \beta), \quad (s \colon \oO_X\to F)
\mapsto F,  
\end{align*}
which is in fact an isomorphism (see e.g. \cite[Lemma 3.2]{FT}). Therefore our descendent $\DT_4$ invariants \eqref{des invs} can be calculated by similar descendent invariants on stable pair moduli spaces. In this case all stable pairs are scheme theoretically supported on 
the zero section $\mathbb{P}^3\subset X$ and the moduli space $P_1(X,3)$ has a natural orientation (up to an overall global sign) coming from $P_1(\mathbb{P}^3,3)$ (see e.g. \cite[Proposition 3.3]{CMT2}). 
Using the 4-fold stable pair vertex \cite{CK2, CKM}, one can explicitly compute the invariant and 
obtain\,\footnote{We are very grateful to Sergej Monavari for sharing his computation to us \cite{M}.}  
$\langle\tau_1([\mathbb{P}^2]) \rangle_{3}=-22610$. 

In this case, the RHS of Conjecture \ref{main conj} is 
\begin{align*}
&\quad \quad \frac{1}{6}n_{0,3}([\mathbb{P}^1])-\frac{1}{3}m_{1,2}-3\,n_{1,3}=\frac{-68060}{6}-\frac{1}{3}\times (-67000)-3\times 11200=-22610,
\end{align*}
which matches with our descendent invariant. Here we have used $n_{1,3}=11200$ 
from~\cite[Table 1]{KP}.
\end{proof}

\providecommand{\bysame}{\leavevmode\hbox to3em{\hrulefill}\thinspace}
\providecommand{\MR}{\relax\ifhmode\unskip\space\fi MR }
\providecommand{\MRhref}[2]{%
  \href{http://www.ams.org/mathscinet-getitem?mr=#1}{#2}}
\providecommand{\href}[2]{#2}

\end{document}